\documentclass[10pt]{paper}%
\usepackage[T1]{fontenc}
\usepackage{amsmath,amssymb}
\usepackage{fullpage}
\usepackage{amsthm}
\usepackage{color}
\usepackage{amsmath}
\usepackage{graphicx}
  \usepackage[all]{xy}
\usepackage{amsfonts}
\usepackage{amssymb}%
\providecommand{\U}[1]{\protect\rule{.1in}{.1in}}
\providecommand{\U}[1]{\protect\rule{.1in}{.1in}}
\providecommand{\U}[1]{\protect\rule{.1in}{.1in}}
\providecommand{\U}[1]{\protect\rule{.1in}{.1in}}
\providecommand{\U}[1]{\protect\rule{.1in}{.1in}}
\newcommand{\ulambda}{{\boldsymbol{\lambda}}}

\newcommand{\unu}{{\boldsymbol{\nu}}}

\newtheorem{Th}{Theorem}[section]
\newtheorem{lemma}[Th]{Lemma}
\newtheorem{Cor}[Th]{Corollary}
\newtheorem{Prop}[Th]{Proposition}

\theoremstyle{remark}
\newtheorem{Rem}[Th]{Remark}{\rmfamily}
\theoremstyle{definition}
\newtheorem{Def}[Th]{Definition}{\rmfamily}
{\rmfamily}

\def\cH{\mathcal{H}}
\def\C{\mathbb{C}}
\def\Z{\mathbb{Z}}

\DeclareRobustCommand\mmodels{\Relbar\joinrel\mathrel{|}}
\def\cP{\mathcal{P}}

\newcommand{\Mat}{\operatorname{Mat}}
\def\pos{\operatorname{pos}}
\def\te{\tilde{e}}
\def\tg{\tilde{g}}
\def\pos{\operatorname{pos}}
\def\te{\widetilde{e}}
\def\tg{\widetilde{g}}
\def\ta{\widetilde{a}}
\begin{document}

\title{Clifford theory for Yokonuma--Hecke algebras and deformation of complex reflection groups}
\author{N. Jacon and L. Poulain d'Andecy }
\maketitle
\date{}

\begin{abstract}
We define and study an action of the symmetric group on the Yokonuma--Hecke algebra. This leads to the definition of two classes of algebras. The first one is connected with the image of the algebra of the braid group inside the Yokonuma--Hecke algebras, and in turn with an algebra defined by Aicardi and Juyumaya known as the algebra of braids and ties. The second one can be seen as new deformations of complex reflection groups of type $G(d,p,n)$.  We provide several presentations for both algebras and a complete study of their representation theories using Clifford theory. 
\end{abstract}

\section{Introduction}

There are several known interesting ways to define an algebra deforming the group algebra  of the complex reflection group of type $G(d,1,n)$ (the wreath product $\mathbb{Z}/d\mathbb{Z}\wr \mathfrak{S}_n$ where $d,n\in \mathbb{N}_{>0}$).  
 The first and maybe the most famous way leads to the ``Ariki--Koike algebra'' whose representation theory has been intensively studied in the last past decades, both in the semisimple and in the  modular setting. Recently, 
 several works have been presented around another deformation of this group algebra: the Yokonuma--Hecke algebra $Y_{d,n}$. One reason of its recent interest comes from knot theory where the Yokonuma--Hecke algebra allows the construction of knots and links invariants (in the same spirit as in the case of the Hecke algebra of type $A$).

The aim of this paper is to define two new remarkable types of algebras using the Yokonuma--Hecke algebra, to study their structures and their representation theories. To do this, we use  a very natural action (by automorphisms) of  the symmetric group $\mathfrak{S}_d$  on the Yokonuma--Hecke algebra $Y_{d,n}$ and 
 we define our algebras to be the algebras of fixed points, first under the action of the full symmetric group $\mathfrak{S}_d$, and second under the action of a certain subgroup $\mathbb{Z}/p\mathbb{Z}$ of $\mathfrak{S}_d$  (where $p$ divides $d$).

We show that, over a suitable base ring, the subalgebra $Y_{d,n}^{\mathfrak{S}_d}$ of fixed points under the action of $\mathfrak{S}_d$ coincides with the image of the group algebra of the braid group inside the Yokonuma--Hecke algebra. In other words, it coincides with the subalgebra generated by the images of the braid generators. In particular, the subalgebra $Y_{d,n}^{\mathfrak{S}_d}$ is connected with an algebra defined in \cite{AJ} by A. Aicardi and J. Juyumaya which has  applications again in knot theory \cite{AJ2}. This algebra, called the algebra of braids and ties $BT_n$,  is defined as an abstract algebra by generators and relations. In \cite{Ry} and \cite{ER}, Espinoza and  Ryom-Hansen have shown that, when $d\geq n$, the algebra $BT_n$ is isomorphic to the subalgebra of $Y_{d,n}$ generated by the images of the braid generators. The authors also studied the representation theory of $BT_n$. They gave a construction of the irreducible representations in the semisimple case using a cellular structure. 

As a consequence of the results of this paper, we obtain that the algebra $BT_n$ is isomorphic to the algebra $Y_{d,n}^{\mathfrak{S}_d}$ of fixed points under the action of the symmetric group. We moreover show that the algebra $BT_n$ is isomorphic to the subalgebra generated by the images of the braid generators if and only if $d\geq n$ (in particular, we give a different proof of the ``if'' part using the description as a subalgebra of fixed points). 
Importantly, this approach allows  to study the representation theory, in both the semisimple and the modular situations, of the algebra $BT_n$ (as a particular case of $Y_{d,n}^{\mathfrak{S}_d}$) in a quite simple way using Clifford theory and the known representation theory of $Y_{d,n}$. This gives a more conceptual proof for 
 the parametrization result of the simple modules obtained in \cite{Ry} in the semisimple case, and a new parametrization result in the non semisimple one.

 The second class of algebras provide deformations of  the  group algebras of complex reflection groups of type $G(d,p,n)$, as we will explain now. 
As noted above, the Yokonuma--Hecke algebra can be seen as a deformation of  the complex reflection groups of type $G(d,1,n)$. The complex reflection groups of this type are particular cases of the unique infinite family $G(d,p,n)$ (where $p$ divides $d$) of irreducible complex reflection groups in the classification by Shephard and Todd. Thus,  it is natural to ask if one can define a new class of algebras deforming the whole infinite series  in the same spirit as  the Yokonuma--Hecke algebra deforms the group algebra of $G(d,1,n)$. We here show that such algebras naturally appear in our setting. To do this, 
  if $p$ is an integer dividing $d$, we consider the action of  the cyclic group $\mathbb{Z}/p\mathbb{Z}$ (as a subgroup of $\mathfrak{S}_d$)  by automorphisms on  $Y_{d,n}$. We then show that the subalgebra of fixed points $Y_{d,n}^{\mathbb{Z}/p\mathbb{Z}}$ can be interpreted as a (new) deformation of the group algebra of the complex reflection group of type $G(d,p,n)$, thus generalizing the known description of $G(d,p,n)$ as a subgroup of $G(d,1,n)$. In particular, for $p=d=2$, we obtain a deformation of the group algebra of the Weyl group of type $D_n$, and we will provide in this case a rather simple explicit presentation. 
  
We study the structure of the algebra $Y_{d,n}^{\mathbb{Z}/p\mathbb{Z}}$ by giving several bases and presentations of it. In particular, we show that this deformation can be described as a quotient of the braid group of type $G(d,p,n)$ over certain relations that we give explicitly. We moreover study its representation theory, again in both the semisimple and modular cases using Clifford theory. We note that, as for the Yokonuma--Hecke algebra in the $G(d,1,n)$ situation, a parametrization of the simple modules is easier to obtain than for the usual Hecke algebras associated to these complex reflection groups. 
     
     The paper will be organized as follows. We describe the action of the symmetric group $\mathfrak{S}_d$ on the Yokonuma--Hecke algebra in Section 2. This action is then intensively used in the rest of the paper.
      In Section 3, we recall several aspects of the representation theory of the Yokonuma--Hecke algebra. We  then use Clifford theory to 
  obtain a parametrization result of the simple modules for  the subalgebra of fixed points of $Y_{d,n}$ under the action of an arbitrary  subgroup of $\mathfrak{S}_d$. 
  The two last parts are then devoted to the study of the subalgebra of fixed points in two particular cases. Section 4 concerns the case of the full symmetric group, while Section 5 deals with the subgroups $\mathbb{Z}/p\mathbb{Z}$. \\
  \\
  {\bf Acknowledgement.} The first author is supported by Agence National de la Recherche Project ACORT ANR-12-JS01-0003.

\section{Yokonuma--Hecke algebras and action of the symmetric group}

Let  $d,n \in\Z_{>0}$ and $q$ be an indeterminate. In this section, we work over the ring $\C[q,q^{-1}]$. We here recall the definition of the Yokonuma--Hecke algebras $Y_{d,n}$ and some of the properties we need in the following. 
 We end the section with the definition of the action of the symmetric group $\mathfrak{S}_d$ on $Y_{d,n}$. 

\subsection{Definitions}

We denote $\mathfrak{S}_n$ the symmetric group on $n$ letters and $\pi_j$, for $j\in \{1,\ldots,n-1\}$ the transposition $(j,j+1)\in \mathfrak{S}_n$. The Yokonuma--Hecke algebra $Y_{d,n}$ is the associative $\C[q,q^{-1}]$-algebra with generators
$$g_1,\ldots,g_{n-1},t_1,\ldots,t_n,$$
subject to the following defining relations (\ref{def-1})--(\ref{def-2}):
\begin{equation}\label{def-1}
\begin{array}{rclcl}
g_ig_j & = & g_jg_i &\quad & \text{for $i,j\in\{1,\ldots,n-1\}$ such that $\vert i-j\vert > 1$\,,}\\[0.2em]
g_ig_{i+1}g_i & = & g_{i+1}g_ig_{i+1} && \text{for $i\in\{1,\ldots,n-2\}$\,,}\\[0.2em]
g_i^2  & = & 1 + (q-q^{-1}) \, e_{i} \, g_i &\quad& \text{for $i\in\{1,\ldots,n-1\}$\,,}
\end{array}
\end{equation}
\begin{equation}\label{def-2}
\hspace{0.0cm}\begin{array}{rclcl}
t_it_j & =  & t_jt_i &\qquad\ \ \,&  \text{for $i,j\in\{1,\ldots,n\}$\,,}\\[0.1em]
g_it_j & = & t_{\pi_i(j)}g_i && \text{for $i\in\{1,\ldots,n-1\}$ and $j\in\{1,\ldots,n\}$\,,}\\[0.1em]
t_j^d   & =  &  1 && \text{for $j\in\{1,\ldots,n\}$\,,}
\end{array}
\end{equation}
where $e_i :=\displaystyle\frac{1}{d}\sum_{1\leq s\leq d}t_i^s t_{i+1}^{-s}$ for $i=1,\ldots,n-1$. The elements $e_i$ are idempotents and we have:
\begin{equation}
g_i^{-1} =  g_i - (q-q^{-1}) e_i  \qquad \text{for all $i=1,\ldots,n-1$\,.}
\end{equation}
We note also the following easily verified characteristic equation for the generators $g_i$:
\begin{equation}\label{char-g}
(g_i^2-1)(g_i^2-(q-q^{-1})g_i-1)=0\ , \qquad \text{for all $i=1,\ldots,n-1$\,.}
\end{equation}
Let $w\in\mathfrak{S}_n$ and $\pi_{i_1}\dots \pi_{i_k}$ a reduced expression for $w$. We define $g_w:=g_{i_1}\dots g_{i_k}$. Then the following set forms a basis of $Y_{d,n}$ \cite{ju2}:
\begin{equation}\label{basis-t}
\{\ t_1^{\alpha_1}\dots t_n^{\alpha_n}g_w\ |\ \alpha_1,\dots,\alpha_n\in\{1,\dots,d\}\,,\ w\in\mathfrak{S}_n\ \}\ .
\end{equation}
As a consequence, the commutative subalgebra $\mathcal{T}_{d,n}:=\langle t_1,\ldots, t_n\rangle$ of $Y_{d,n}$ generated by $t_1,\dots,t_n$ is isomorphic to the group algebra of $(\mathbb{Z}/d\mathbb{Z})^n$ over $\C[q,q^{-1}]$.

\subsection{Another presentation of $Y_{d,n}$}

\paragraph{Ordered partitions of $\{1,\dots,n\}$.} For the definition below, note that for $I=(I_1,\dots,I_d)\in\cP_d(n)$, some of the subsets $I_1,\dots,I_d$ are allowed to be empty. We emphasize that the ordering of the subsets in $I$ is relevant.
\begin{Def} An \emph{ordered partition} of $\{1,\dots,n\}$ is a multiplet $(I_1,\dots,I_d)$ of $d$ pairwise disjoint subsets of $\{1,\dots,n\}$ such that $I_1\cup\dots\cup I_d=\{1,\dots,n\}$.
We denote by $\cP_d(n)$ the set of ordered partitions of $\{1,\dots,n\}$ into $d$ parts. 
\end{Def}
Let $I=(I_1,\dots,I_d)\in\cP_d(n)$ and $j\in\{1,\dots,n\}$. We define the ``position'' of $j$ in $I$ by
\begin{equation}\label{pos}
\pos_j(I):=a\,,\ \ \ \ \ \text{if $j\in I_a$.}
\end{equation}

Let $I=(I_1,\dots,I_d)\in\cP_d(n)$. We consider the following two natural actions of symmetric groups $\mathfrak{S}_n$ and $\mathfrak{S}_d$  on the set $\cP_d(n)$:
\begin{itemize}
\item Let $\pi\in\mathfrak{S}_n$. We set:
\[\pi(I):=\bigl(\pi(I_1),\dots,\pi(I_d)\bigr)\ ,\]
where $\pi(S):=\{\pi(s)\ |\ s\in S\}$ for any $S\subset\{1,\dots,n\}$.
\item  Let $\sigma\in\mathfrak{S}_d$. We set:
\[I^{\sigma}:=\bigl(I_{\sigma^{-1}(1)},\dots,I_{\sigma^{-1}(d)}\bigr)\ .\]
\end{itemize}

We note that, for any $i\in\{1,\dots,n\}$ and $\sigma\in\mathfrak{S}_d$, we have (recall that $\pos_i(I)\in\{1,\dots,d\}$ by definition):
\begin{equation}\label{form-pos}
\pos_i(I^{\sigma})=\sigma\bigl(\pos_i(I)\bigr)\ .
\end{equation}

\paragraph{Characters of $\mathcal{T}_{d,n}$.} 

Recall that the subalgebra $\mathcal{T}_{d,n}=\langle t_1,\dots,t_n\rangle$ of $Y_{d,n}$ is isomorphic to the group algebra of $(\Z/d\Z)^n$. We will implicitly make and use this identification in all the following.

Let $\{\xi_1,\dots,\xi_d\}$ be the set of complex roots of unity of order $d$
 such that $\xi_1:=\textrm{exp}(2i\pi/d)$ and $\xi_j=\xi_1^j$ for $j=1,\ldots,d$. 
 An irreducible  complex character $\chi$ of the subalgebra $\mathcal{T}_{d,n}$ is determined by the choice of $\chi (t_j)\in \{\xi_1,\ldots ,\xi_d \}$ for each $j=1,\ldots,n$. 

The set of irreducible  complex characters of $\mathcal{T}_{d,n}$ is in bijection with the set $\cP_d(n)$. The bijection is given by associating to a character $\chi$ the ordered partition $I(\chi)\in\cP_d(n)$ defined by
\begin{equation}\label{def-Ichi}
I(\chi)=(I(\chi)_1,\dots,I(\chi)_d)\,,\ \ \ \ \ \ \text{where}\ I(\chi)_a:=\{j\in \{1,\ldots,n\}\ |\ \chi (t_j)=\xi_a\}\ .
\end{equation}
Furthermore, the symmetric group $\mathfrak{S}_n$ acts naturally on the set of irreducible complex characters of  $\mathcal{T}_{d,n}$. For $\pi\in\mathfrak{S}_n$, the character $\pi(\chi)$ is defined by $\pi(\chi)\bigl(t_j\bigr):=\chi(t_{\pi^{-1}(j)})$, for $j=1,\dots,n$. It is straightforward to check that this action corresponds to the action of $\mathfrak{S}_n$ on $\cP_d(n)$ given above in the following sense: the character $\pi(\chi)$ corresponds to $\pi(I)$ if $\chi$ corresponds to the partition $I$ (that is, $\pi\bigl(I(\chi)\bigr)=I\bigl(\pi(\chi)\bigr)$).

Now, for $\chi$ an irreducible  complex character of $\mathcal{T}_{d,n}$, let $E_{\chi}$ be the corresponding primitive idempotent. A formula for $E_{\chi}$ is:
\begin{equation}\label{form-E}
E_{\chi} =\prod_{1\leq i\leq n}\Bigl(\frac{1}{d}\sum_{1\leq s \leq d}\chi(t_i)^st_i^{-s}\Bigr)\ .
\end{equation}
The primitive idempotents $E_{\chi}$, where $\chi$ runs over the set of irreducible complex characters, form a complete set of pairwise orthogonal idempotents of $\mathcal{T}_{d,n}$. In particular, they form a basis of the subalgebra $\mathcal{T}_{d,n}$.

Therefore, for $i\in\{1,\dots,n-1\}$, the element $e_i$ can be expressed in terms of the primitive idempotents $E_{\chi}$. Explicitly, since we have $\chi(e_i)=1$ if $\chi(t_i)=\chi(t_{i+1})$ and $0$ otherwise, we obtain
\begin{equation}\label{form-e}
e_i=\sum_{\chi(t_i)=\chi(t_{i+1})}  E_{\chi}\ ,
\end{equation}
where the sum is over all irreducible complex characters $\chi$ such that $\chi(t_i)=\chi(t_{i+1})$.

\paragraph{Another presentation of $Y_{d,n}$.} For $I\in\cP_d(n)$, we define $E_I:=E_{\chi}\,$, where $\chi$ is the character such that $I=I(\chi)$. From the knowledge of the basis (\ref{basis-t}) of $Y_{d,n}$, it follows that the following set also forms a basis of $Y_{d,n}$:
\begin{equation}\label{basis-E}
\{\ E_Ig_w\ |\ I\in\cP_d(n)\,,\ w\in\mathfrak{S}_n\ \}\ .
\end{equation}
In particular, the algebra $Y_{d,n}$ is generated by the elements
\begin{equation}\label{gen2}
g_1,\dots,g_{n-1}\ \ \ \ \text{and}\ \ \ \ E_I\,,\ \ I\in\cP_d(n)\ .
\end{equation}
A set of defining relations is (\ref{def-1}) together with:
\begin{equation}\label{def-2b}
\hspace{-1.2cm}\begin{array}{rclcl}
E_IE_J & =  & \delta_{I,J}E_I &\qquad&  \text{for $I,J\in\cP_d(n)$\,,}\\[0.1em]
g_iE_I & = & E_{\pi_i(I)}g_i && \text{for $i\in\{1,\ldots,n-1\}$ and $I\in\cP_d(n)$\,,}\\[0.1em]
\sum_{I\in\cP_d(n)} E_I & = & 1\ , &&
\end{array}
\end{equation}
where, following (\ref{form-e}), we have now, for $i=1,\dots,n-1$,
\begin{equation}\label{form-e2}
e_i=\sum_{\text{\scriptsize{$\begin{array}{c}I\in\cP_d(n) \\
\pos_i(I)=\pos_{i+1}(I)
\end{array}$}}}\!\!\!\!E_I\ ,
\end{equation}
where $\pos_j(I)$ is the position defined in (\ref{pos}). Relations (\ref{def-2b}) are obviously satisfied in $Y_{d,n}$\,. The stronger assertion that they form with (\ref{def-1}) a set of defining relations follows easily from the fact that they allow to write any product of the generators (\ref{gen2}) as a linear combination of the basis elements (\ref{basis-E}).

\subsection{Action of $\mathfrak{S}_d$ on $Y_{d,n}$}\label{ect}

Let $\sigma\in\mathfrak{S}_d$. We define a linear map $Y_{d,n}\ni x\mapsto x^{\sigma}\in Y_{d,n}$ by extending linearly the following formula:
\begin{equation}\label{action}
(E_Ig_w)^{\sigma}:=E_{I^{\sigma}}g_w\,,\ \ \ \quad\text{for $I\in\cP_d(n)$ and $w\in\mathfrak{S}_n$.}
\end{equation}
\begin{Prop} The map (\ref{action}) induces  an action of $\mathfrak{S}_d$ on $Y_{d,n}$ by automorphisms of algebras.
\end{Prop}
\begin{proof} Let $\sigma\in\mathfrak{S}_d$. The map (\ref{action}) is given on the generators of $Y_{d,n}$ by
\begin{equation}\label{action-gen}
E_I^{\sigma}:=E_{I^{\sigma}}\,,\ \ \text{for $I\in\cP_d(n)$,}\ \ \ \ \ \text{and}\ \ \ \ \ g_i^{\sigma}:=g_i\,,\ \ \text{for $i=1,\dots,n-1$.}
\end{equation}
We shall check that the map $x\mapsto x^{\sigma}$ is an algebra automorphism of $Y_{d,n}$, by verfying  that the  formulas (\ref{action-gen}) preserve the defining relations (\ref{def-1}) and (\ref{def-2b}). 

We first consider  the defining relations (\ref{def-1}). For the braid relations, there is nothing to check. For the quadratic relation, it is enough to note that $e_i^{\sigma}=e_i$. This follows directly from Formula (\ref{form-e2}) together with the obvious fact that $\pos_i(I^{\sigma})=\pos_{i+1}(I^{\sigma})$ if and only if $\pos_i(I)=\pos_{i+1}(I)$. Now, it remains to consider the defining relations (\ref{def-2b}).
\begin{itemize}
\item First, we note that, for $I,J\in\cP_d(n)$, we have $I^{\sigma}=J^{\sigma}$ if and only if $I=J$ (all the stabilizers of the action of $\mathfrak{S}_d$ on $\cP_d(n)$ are trivial). This gives  the first line of (\ref{def-2b}).
\item For the second line, it is enough to remark that the two actions of $\mathfrak{S}_n$ and $\mathfrak{S}_d$ on $\cP_d(n)$ commute, and therefore $\pi_i(I^{\sigma})=\pi_i(I)^{\sigma}$ for any $i\in\{1,\ldots,n-1\}$ and $I\in\cP_d(n)$.
\item For the third line of (\ref{def-2b}), we have obviously $\sum_{I\in\cP_d(n)} E_{I^{\sigma}}=\sum_{I\in\cP_d(n)} E_I$.
\end{itemize}
The proof of the proposition is concluded by noticing that Formulas (\ref{action}), for $\sigma \in\mathfrak{S}_d$, indeed define an action of $\mathfrak{S}_d$ on $Y_{d,n}$ since $I\mapsto I^{\sigma}$ is an action of $\mathfrak{S}_d$ on $\cP_d(n)$.
\end{proof}

For any subgroup $G$ of $\mathfrak{S}_d$, 
 $G$ acts on $Y_{d,n}$ by automorphisms.
 We denote $Y_{d,n}^{G}$ the subalgebra of $Y_{d,n}$ of fixed points under the action of $G$. Namely, the subalgebra $Y_{d,n}^{G}$ is formed by elements $x$ of $Y_{d,n}$ satisfying $x^{\sigma}=x$ for any $\sigma\in G$.

\section{Representation theory  of $Y_{d,n}$}\label{si}
In this part, we study the representation theory of the Yokonuma--Hecke algebra using an isomorphism theorem that we first recall. Let $\theta : \mathbb{C}[q,q^{-1}] \to \mathbb{C}$ be a specialization. Our results will concern the algebras $\mathbb{C}_{\theta} Y_{d,n}:=\mathbb{C} \otimes_{\mathbb{C}[q,q^{-1}]} Y_{d,n}$
 and $\mathbb{C}_{\theta}Y_{d,n}^{G}:=\mathbb{C}\otimes_{\mathbb{C}[q,q^{-1}]} Y_{d,n}^G$ for $G$, an arbitrary  subgroup of $\mathfrak{S}_d$.
 
\subsection{An isomorphism theorem}
\paragraph{Compositions of $n$.} Let $\operatorname{Comp}_d (n)$ be the set of {\it $d$-compositions} of $n$, that is the set of $d$-tuples $\mu=(\mu_1,\ldots,\mu_d)\in\Z_{\geq0}^d$ such that $\sum_{1\leq a\leq d} \mu_a =n$. We denote $\mu\mmodels_d n$.

For $\mu\mmodels_d n$, the Young subgroup $\mathfrak{S}^{\mu}$ is the subgroup $\mathfrak{S}_{\mu_1}\times\dots\times\mathfrak{S}_{\mu_d}$ of  $\mathfrak{S}_{n}$, where $\mathfrak{S}_{\mu_1}$ acts on the letters $\{1,\dots,\mu_1\}$, $\mathfrak{S}_{\mu_2}$ acts on the letters $\{\mu_1+1,\dots,\mu_2\}$, and so on.

For $\mu=(\mu_1,\ldots,\mu_d)\mmodels_d n$, we define $\cP(\mu)$ to be the subset of $\cP_d(n)$ consisting of ordered partitions $(I_1,\dots,I_d)$ satisfying in addition $|I_a|=\mu_a$ for $a=1,\dots,d$. 

Finally, the natural action of $\mathfrak{S}_d$ on the set of compositions of $n$ into $d$ parts is given by 
$$\mu^{\sigma}:=(\mu_{\sigma^{-1}(1)},\dots,\mu_{\sigma^{-1}(d)})$$
 if $\mu=(\mu_1,\dots,\mu_d)\mmodels_d n$. We then have naturally $I^{\sigma}\in\cP(\mu^{\sigma})$ if $I\in\cP(\mu)$. We set
\[m_{\mu}:=\vert \cP(\mu)\vert=\frac{n!}{\mu_1!\mu_2!\dots\mu_d!}\ .\]

\paragraph{Hecke algebra.} By definition, the Iwahori--Hecke algebra $\mathcal{H}_n$ over $\mathbb{C}[q,q^{-1}]$ has a presentation by generators $T_1,\ldots,T_{n-1}$ and relations 
\begin{equation}
\begin{array}{rclcl}
T_iT_j & = & T_jT_i &\quad & \text{for $i,j\in\{1,\ldots,n-1\}$ such that $\vert i-j\vert > 1$\,,}\\[0.2em]
T_iT_{i+1}T_i & = & T_{i+1}T_iT_{i+1} && \text{for $i\in\{1,\ldots,n-2\}$\,,}\\[0.2em]
T_i^2  & = & 1 + (q-q^{-1})T_i &\quad& \text{for $i\in\{1,\ldots,n-1\}$\,,}
\end{array}
\end{equation}
If $\mu\mmodels_d n$, we denote by $\mathcal{H}^{\mu}$ the algebra $\mathcal{H}_{\mu_1}\otimes \ldots \otimes  \mathcal{H}_{\mu_d}$. As usual, if $\theta : \mathbb{C}[q,q^{-1}] \to \mathbb{C}$ is a specialization, we denote by  $\mathbb{C}_{\theta} \mathcal{H}_{n}$ and 
 $\mathbb{C}_{\theta} \mathcal{H}^{\mu}$ the associated specialized algebras.

\paragraph{Isomorphism.} In \cite{JPA}, an algebra isomorphism 
\begin{equation}\label{iso}
\Psi_{d,n}\ :\ Y_{d,n}\ \to\ \bigoplus_{\mu \mmodels_d n} \Mat_{m_{\mu}} (\mathcal{H}^{\mu})
\end{equation}
is given. Such an isomorphism has been previously obtained in \cite{Lu}, even in a more general setting.  For $x\in Y_{d,n}$, the element $\Psi_{d,n}(x)$ is thus a collection of matrices indexed by compositions of $n$ into $d$ parts. Moreover, rows and columns of a matrix in $\Mat_{m_{\mu}} ( \mathcal{H}^{\mu})$ are indexed by the set $\cP(\mu)$. Thus, the element $\Psi_{d,n}(x)$ is given by:
$$\left( (\Psi_{d,n}(x)_{I,J})_{I,J\in\cP(\mu)}  \right)_{\mu\mmodels_d n}.$$

We will give the images by $\Psi_{d,n}$ of the generators $E_I$, $I\in\cP_d(n)$, and $g_1,\dots,g_{n-1}$, by giving their non-zero coefficients.

First, let $I\in\cP_d(n)$. Let $\mu\mmodels_d n$ be the composition such that $I\in\cP(\mu)$. Then, the only non-zero coefficient in $\Psi_{d,n}(E_I)$ is 
\[\Psi_{d,n}(E_I)_{I,I}=1\ ;\]
In other words, the image of $E_I$ is the elementary matrix in $\Mat_{m_{\mu}}\bigl(\mathcal{H}^{\mu}\bigr)$ with $1$ in the diagonal position corresponding to $I$.

Let $i\in\{1,\dots,n-1\}$ and recall that $\pi_i:=(i,i+1)$. The non-zero coefficients in $\Psi_{d,n}(g_i)$ are
\begin{equation}\label{iso-gi}
\Psi_{d,n}(g_i)_{I,\pi_i(I)}= \left\{\begin{array}{lll}
1 && \text{if $\pi_i(I)\neq I$\,,}\\[0.5em]
1\otimes \dots 1\otimes T_k\otimes 1\dots \otimes 1\quad && \text{if $\pi_i(I) = I$\,,}
\end{array}\right.
\end{equation}
where $I$ varies in $\cP_{d}(n)$, and, in the second line above, $T_k$ is in the $a$-th factor if $\pos_i(I)=a$ and $k$ is determined by $k=\sharp\{j\in I_a\ |\ j\leq i\}$.

\subsection{Representation theory} \label{subsec-rep}
\paragraph{Simple modules.} From the isomorphism (\ref{iso}), we have that any simple module for $\mathbb{C}_{\theta} Y_{d,n}$ is of the form
\begin{equation}\label{simple}
V=(V_1\otimes \ldots \otimes V_d)^{m_{\mu}}=(V_1\otimes \ldots \otimes V_d)\otimes W_{\mu} ,
\end{equation}
where $V_a$ is an irreducible module for $\mathbb{C}_{\theta}\cH_{\mu_a} $ for each $a\in\{1,\dots,d\}$, and $W_{\mu}$ is a vector space with basis $\{w_I\}_{I\in\cP(\mu)}$.

Let $(v_1\otimes\dots\otimes v_d)\otimes w_I$ be an arbitrary vector in $V$. Then the action of the generators of $\mathbb{C}_{\theta} Y_{d,n}$ is given as follows:
\begin{equation}\label{act-E}
E_I\bigl((v_1\otimes\dots\otimes v_d)\otimes w_I\bigr)=(v_1\otimes\dots\otimes v_d)\otimes w_I\ \ \quad\text{and}\quad\ \ E_J\bigl((v_1\otimes\dots\otimes v_d)\otimes w_I\bigr)=0\ \ \ \ \text{if $J\neq I$\,,}
\end{equation}
\begin{equation}\label{act-g}
g_i\bigl((v_1\otimes\dots\otimes v_d)\otimes w_I\bigr)=\left\{\begin{array}{lll}
(v_1\otimes\dots\otimes v_d)\otimes w_{\pi_i(I)} && \text{if $\pi_i(I)\neq I$\,,}\\[0.5em]
(v_1\otimes\dots\otimes T_k(v_a)\otimes\dots\otimes v_d)\otimes w_I && \text{if $\pi_i(I) = I$\,,}
\end{array}\right.
\end{equation}
where $a$ and $k$ are determined as in (\ref{iso-gi}).
\paragraph{Parametrization.}
   
Let us now study more precisely the representation theory of $Y_{d,n} $ using the isomorphism (\ref{iso}). This will be needed in the following. We also fix several useful additional notations. 
 To this extent, consider the Hecke algebra $\mathbb{C}_{\theta} \mathcal{H}_n$. Let us  denote :
 $$e:=
 \textrm{min} \{i>0\ |\ 1+\theta (q)^2+\ldots +\theta (q)^{2i-2}=0\}.$$
(we set $e=\infty$ if no such integer exist)

 Then, by Dipper-James (see \cite[Thm. 3.5.14]{GJ}) the simple $\mathbb{C}_{\theta} \mathcal{H}_n$-modules are labeled by 
  the set of $e$-regular partitions of rank $n$ which we denote by $\Lambda_n^e$ (with the convention that 
 $\Lambda_n^\infty=:\Lambda_n$ is the set of  all partitions of rank $n$). We thus have:
 $$\textrm{Irr} (\mathbb{C}_{\theta} \mathcal{H}_n)=\{ V^{\lambda} \ |\ \lambda\in \Lambda_n^{e} \}.$$ 
Now for $\mu=(\mu_1,\ldots,\mu_d)\in \operatorname{Comp}_d (n)$, let us denote 
$$\Lambda^e_{\mu}:=\{\ulambda=(\lambda^1,\ldots,\lambda^d)\ |\forall i\in\{1,\ldots,d\},\ \lambda^i\in \Lambda^e_{\mu_i}\}$$
For $\ulambda=(\lambda^1,\ldots,\lambda^d)\in \Lambda^e_{\mu}$, we also set :
  $$V^{\ulambda}=(V^{\lambda^1}\otimes \ldots \otimes V^{\lambda^d})\otimes W_{m_{\mu}}.$$
   In our situation, using the isomorphism (\ref{iso}), we get (see \cite[\S 4.1]{JPA}):
 $$\textrm{Irr} ( \mathbb{C}_{\theta} {Y}_{d,n} )=
  \{ V^{\ulambda}\ |\ \ulambda\in \Lambda_{\mu}^{e} ,\ \mu\in \operatorname{Comp}_d (n)\}.$$
  We denote by $y (\ulambda,e)$ the dimension of $V^{\ulambda}$ (it indeed only depends on $e$ and $\ulambda$). We have:
  $$y (\ulambda,e)=m_{\mu} .\prod_{1\leq i\leq d} y (\lambda^i,e ),$$
  where $\mu=(|\lambda^1|,\ldots,|\lambda^d|)$ and, for $i=1,\ldots,d$, the number  $y (\lambda^i,e)$ 
   is the dimension of the simple $\mathcal{H}_{\mu_i}$-module $V^{\lambda^i}$. 
Note that  if $e\in \mathbb{N}$, this dimension may be recursively computed using the LLT algorithm and the connection with the theory of canonical bases for Fock spaces (we refer to \cite[Ch. 4, 5]{GJ} for details.) 
In the case where $e=\infty$, all  partitions are $e$-regular. Moreover, 
 the dimension is   given in thise case by $n!$ divided by the product of the hook lengths of the 
  Young diagrams of the $\lambda^i$'s.

 \paragraph{Relation with cyclotomic Ariki-Koike algebras.}\label{iso?}
 The algebra $Y_{d,n}$ is a deformation of the complex reflection group of type $G(d,1,n)$ over $\mathbb{C}[q,q^{-1}]$. On the other hand, there exists 
  another well known one-parameter deformation of   $G(d,1,n)$ over $\mathbb{C}[q,q^{-1}]$: the cyclotomic (specialization of) Ariki--Koike algebra (see \cite[\S 5.5]{GJ}). Let $(m_1,\ldots,m_d)\in \mathbb{N}^d$ and set $\eta_d:=\textrm{exp}(2i\pi/d)$. The cyclotomic Ariki--Koike algebra $\mathbb{H}_{d,n}$ has a presentation by generators $T_0$, $T_1$, \ldots, $T_{n-1}$ and relations :
  \begin{equation}\label{defAK-1}
\begin{array}{rclcl}
T_iT_j & = & T_jT_i &\quad & \text{for $i,j\in\{1,\ldots,n-1\}$ such that $\vert i-j\vert > 1$\,,}\\[0.2em]
T_iT_{i+1}T_i & = & T_{i+1}T_iT_{i+1} && \text{for $i\in\{1,\ldots,n-2\}$\,,}\\[0.2em]
T_0 T_1 T_0 T_1 &=&T_1 T_0 T_1 T_0, & & \\
T_i^2  & = & 1 + (q-q^{-1}) \,  T_i &\quad& \text{for $i\in\{1,\ldots,n-1\}$\,,}
\end{array}
\end{equation}
together with the relation
$$(T_0-\eta_d^0 q^{2m_1}) (T_0-\eta_d^{1} q^{2m_2})
\ldots (T_0-\eta_d^{d-1} q^{2m_d})=0$$

 It is natural to ask if $Y_{d,n}$ and $\mathbb{H}_{d,n}$  are isomorphic over $\mathbb{C}[q,q^{-1}]$. They are both free with the same rank and moreover, with the specialization $\theta_1:\mathbb{C}[q,q^{-1}]\to \mathbb{C}$ such that $\theta_1 (q)=\pm 1$, we get two algebras
   $\mathbb{C}_{\theta_1} Y_{d,n}$ and $\mathbb{C}_{\theta_1} \mathbb{H}_{d,n}$ which are both isomorphic to the group algebra of the complex reflection group of type $G(d,1,n)$ over $\mathbb{C}$. 
 
To answer this question, we can consider a specialization $\theta : \mathbb{C}[q,q^{-1}] \to \mathbb{C}$ such that $\theta (q)^2$ is a primitive root of unity  of order $d>1$ and remark that the number of simple modules for both algebras are different in general.
 More precisely, thanks to Ariki's theory and the connections with the crystal basis theory for quantum groups,  several possible parametrizations for the set of  simple  $\mathbb{C}_{\theta} \mathbb{H}_{d,n}$-modules  are available.  One can use the parametrization obtained in \cite[\S 5]{GJ} to deduce that 
 the parametrization set of the simple modules of $\mathbb{C}_{\theta} Y_{d,n}$ that we get here  is strictly contained in the one obtained in   for $\mathbb{C}_{\theta} \mathbb{H}_{d,n}$ in the case where $n\geq d$.  As a  consequence, the two algebras  $ Y_{d,n}$ and $ \mathbb{H}_{d,n}$ cannot be isomorphic.

\paragraph{Action of $\mathfrak{S}_d$ on representations.} The action of $\mathfrak{S}_d$ on $Y_{d,n}$ is translated as an action of $\mathfrak{S}_d$ on the set of representations of $\mathbb{C}_{\theta} Y_{d,n}$. Let $\sigma\in\mathfrak{S}_d$ and $\rho\ :\ \mathbb{C}_{\theta} Y_{d,n}\to \text{End}_{\mathbb{C}} (V)$ be a representation of $\mathbb{C}_{\theta} Y_{d,n}$. Then $V^{\sigma}$ is the $\mathbb{C}_{\theta} Y_{d,n}$-module with the same underlying vector space as $V$ and with action given by $\rho^{\sigma}(x):=\rho(x^{\sigma})$ for $x\in \mathbb{C}_{\theta} Y_{d,n}$.

Let $V$ be a simple module as in (\ref{simple}) and let $\sigma\in\mathfrak{S}_d$. We define a linear map $\Phi_{\sigma}$ from the vector space $V$ to the vector space $ (V_{\sigma^{-1}(1)}\otimes \ldots \otimes V_{\sigma^{-1}(d)})\otimes W_{\mu^{\sigma}}\,$ by
\begin{equation}\label{act-Sd1}
\Phi_{\sigma}\ :\ (v_1\otimes\dots\otimes v_d)\otimes w_I \mapsto (v_{\sigma^{-1}(1)}\otimes \ldots \otimes v_{\sigma^{-1}(d)})\otimes w_{I^{\sigma}}\ .
\end{equation}
We note that $\Phi_{\sigma}\circ \Phi_{\sigma'}=\Phi_{\sigma\sigma'}$. 
 
\begin{lemma}
The map $\Phi_{\sigma}$ gives the following isomorphism of $\mathbb{C}_{\theta} Y_{d,n}$-modules:
\begin{equation}\label{act-Sd2}
V^{\sigma^{-1}}\cong (V_{\sigma^{-1}(1)}\otimes \ldots \otimes V_{\sigma^{-1}(d)})\otimes W_{\mu^{\sigma}}\ ,
\end{equation}
\end{lemma}
\begin{proof} We need to prove that, for $x\in \mathbb{C}_{\theta} Y_{d,n}$ and $v\in V$, we have
\begin{equation}\label{eq-iso}
\Phi_{\sigma}(x^{\sigma^{-1}}\!\!\cdot v)=x\cdot\Phi_{\sigma}(v)\ .
\end{equation}
 Recall that on the generators of $\mathbb{C}_{\theta} Y_{d,n}$, we have $E_I^{\sigma}=E_{I^{\sigma}}$ and $g_i^{\sigma}=g_i$. Then, from Formulas (\ref{act-E})-(\ref{act-g}) giving the action of the generators on $V$, it is straightforward to check that (\ref{eq-iso}) is satisfied. 
\end{proof}

\subsection{Clifford theory}\label{clifford}

Let $G$ be a subgroup of $\mathfrak{S}_d$. For each $V\in \operatorname{Irr} (\mathbb{C}_{\theta}Y_{d,n})$, we set 
$$\mathfrak{H}_{V} (G):=\{ g\in G \ |\ V^{g}\simeq V\}$$ 
to be the inertia subgroup of $V$ with respect to $G$. Note that for $\sigma\in\mathfrak{S}_d$, we have $\mathfrak{H}_{V} (G)=\sigma\mathfrak{H}_{V^{\sigma}} (G)\sigma^{-1}$, and so, in particular, we have $\mathfrak{H}_{V} (G)\cong\mathfrak{H}_{V^{\sigma}} (G)$.

The maps $\Phi_{\sigma}$, with $\sigma\in\mathfrak{H}_{V} (G)$, studied above provide a representation of $\mathfrak{H}_{V} (G)$ on $V$, and moreover, this action of $\mathfrak{H}_{V} (G)$ commutes with the action of $\mathbb{C}_{\theta} Y^G_{d,n}$; this follows from (\ref{eq-iso}).

The following proposition is a formulation in our situation of a more general result which can be found in \cite[Th. A.13]{RR}.
   \begin{Prop}\label{RR}
   For all $V\in \operatorname{Irr} (\mathbb{C}_{\theta}Y_{d,n})$, we have a decomposition as 
    a  $\mathbb{C}_{\theta} Y_{d,n}^G\otimes\C\mathfrak{H}_V(G)$-module:
   $$V=\bigoplus_{M\in \operatorname{Irr} (\mathfrak{H}_{V} (G))} V_M \otimes M\ ,$$
and the non zero $V_M$ give a complete set of  simple $\mathbb{C}_{\theta} Y_{d,n}^{G}$-modules. Moreover, for $(V,W)\in \operatorname{Irr} (\mathbb{C}_{\theta}Y_{d,n})^2 $, $M\in   \operatorname{Irr} (\mathfrak{H}_{V} (G))$ and $M'\in   \operatorname{Irr} (\mathfrak{H}_{W} (G))$, we have $V_M\simeq W_{M'}$ if and only if, first, there exists $g\in G$ such that $V\simeq W^g$  and second, identifying $\mathfrak{H}_{V} (G)$ with $\mathfrak{H}_{W} (G)$, we have $M \simeq M'$. 
   \end{Prop}
   
Hence, to obtain a concrete description of the simple 
$\mathbb{C}_{\theta}Y^G_{d,n}$-modules, it is now necessary to understand which 
 modules $V_M$ are non zero. In our case, the situation is quite simple as we can notice in the following proposition.

\begin{Prop}\label{Cliff}
Keeping the above notations, for all $V\in \operatorname{Irr} (\mathbb{C}_{\theta}Y_{d,n})$ 
 and $M\in \operatorname{Irr} (\mathfrak{H}_{V} (G))$, we have $V_M\neq 0$. 
 In addition, we have:
$$\operatorname{dim}_{\mathbb{C}} (V_M)=
\frac{ \operatorname{dim}_{\mathbb{C}} (V).\operatorname{dim}_{\mathbb{C}} (M)}{ \sharp \mathfrak{H}_{V} (G) }\ .$$
\end{Prop}    
   \begin{proof}
 Let $V=(V_1\otimes\dots V_d)\otimes W_{\mu}\in \operatorname{Irr} (\mathbb{C}_{\theta}Y_{d,n})$ as in (\ref{simple}). Let $v=(v_1\otimes\dots\otimes v_d)\otimes w_I\in V$. We recall that the action $I\mapsto I^{\sigma}$ of $\mathfrak{S}_d$ is is semiregular (that is, all the stabilizers are trivial). Therefore, it follows that such a vector $v$ has a trivial stabilizer under the action of $\mathfrak{H}_{V} (G)$ described in (\ref{act-Sd1}). Denoting $X_v$ the subspace of $V$ generated by al the $\{\Phi_{\sigma}(v)\}_{\sigma\in\mathfrak{H}_{V} (G)}$, we obtain that $X_v$ is isomorphic as a $\mathfrak{H}_{V} (G)$-module to the regular representation.

As a consequence of the preceding discussion, we have that $V$ is isomorphic, as a  $\mathfrak{H}_{V} (G)$-module,  to a direct sum of finitely many copies of the regular representation. The number of copies is obviously $\displaystyle\frac{\operatorname{dim}_{\mathbb{C}} (V)}{\sharp \mathfrak{H}_{V} (G)}$\ . This shows that all the modules $V_M$ are non zero and moreover that the multiplicity of 
 $M$ in the decomposition of $V$ as a $\mathfrak{H}_{V} (G)$-module is  
 $$\frac{\operatorname{dim}_{\mathbb{C}} (V).\textrm{dim}_{\mathbb{C}} (M)}{\sharp \mathfrak{H}_{V} (G)}\ .$$
This proves the formula for $\operatorname{dim}_{\mathbb{C}} (V_M)$ and concludes the proof.
   \end{proof}

   \section{Action of the symmetric group}
    
In this section, we consider the action of the full  symmetric group $\mathfrak{S}_d$ on 
  $Y_{d,n}$. We  consider  the algebra of fixed points under this action,
   study its structure and its relation with the algebra of braid and ties.

\subsection{Subalgebra $Y_{d,n}^{\mathfrak{S}_d}$}

Let $\cP_d(n)/\mathfrak{S}_d$ denote the set of orbits of $\cP_d(n)$ under the action of $\mathfrak{S}_d$. By definition of $\cP_d(n)$ and of the action of $\mathfrak{S}_d$ on it, the orbits are in bijection with the (unordered) partitions of $\{1,\dots,n\}$ into $d$ parts:
\[\cP_d(n)/\mathfrak{S}_d=\bigl\{\ \{I_1,\dots,I_d\}\ \text{such that}\ I_1\cup\dots\cup I_d=\{1,\dots,n\}\ \bigr\}\ .\]
Note again that some parts among $I_1,\dots,I_d$ can be empty. Denote 
$$B_d(n):=\vert\cP_d(n)/\mathfrak{S}_d\vert\ ,$$
the cardinal of $\cP_d(n)/\mathfrak{S}_d$, which is the number of partitions of $\{1,\dots,n\}$ into at most $d$ parts. 

Let $[I]\in \cP_d(n)/\mathfrak{S}_d$ be the orbit of some element $I\in\cP_d(n)$. We set
\[E_{[I]}:=\sum_{\sigma\in\mathfrak{S}_d} E_{I^{\sigma}}\ .\]
From Formula (\ref{action}) giving the action of $\mathfrak{S}_d$ on $Y_{d,n}$, it follows at once that
\begin{equation}\label{basis-inv}
\{\ E_{[I]}\,g_w\ |\ [I]\in\cP_d(n)/\mathfrak{S}_d\,,\ w\in\mathfrak{S}_n\ \}
\end{equation}
is a $\C[q,q^{-1}]$-basis of the subalgebra $Y_{d,n}^{\mathfrak{S}_d}$. In particular, $Y_{d,n}^{\mathfrak{S}_d}$ is free of rank:
\begin{equation}\label{dim-inv}
\text{rk}(Y_{d,n}^{\mathfrak{S}_d})=B_d(n)\, n!\ .
\end{equation}

\paragraph{Elements $e_{[I]}$.} We introduce some notations. First, for any $i,j\in\{1,\dots,n\}$, let
\begin{equation}\label{def-eij}
e_{i,j}:=\frac{1}{d}\sum_{1\leq s\leq d}t_i^s t_{j}^{-s}\ ,
\end{equation}
so that we have in particular $e_i=e_{i,i+1}$ for $i=1,\dots,n-1$. Note also that $e_{i,i}=1$ by definition, and that these elements $e_{i,j}$ commute with each other.
Then, for $I\in\cP_d(n)$, we set
\begin{equation}\label{def-eI}
e_{[I]}:=\prod_{\text{\scriptsize{$\begin{array}{c}i,j=1,\dots,n \\
\pos_i(I)=\pos_{j}(I)
\end{array}$}}}\hspace{-0.4cm}e_{i,j}\hspace{0.4cm}\cdot\hspace{-0.4cm}\prod_{\text{\scriptsize{$\begin{array}{c}i,j=1,\dots,n \\
\pos_i(I)\neq\pos_{j}(I)
\end{array}$}}}\hspace{-0.4cm}(1-e_{i,j})\ .
\end{equation}
The element $e_{[I]}$ depends only on the class $[I]$ since it depends only on the subsets $I_1,\dots,I_d$ forming the ordered partition $I$ and not on their ordering.

\begin{lemma}\label{lem-E-e}
Let $I\in\cP_d(n)$. We have:
\[e_{[I]}=E_{[I]}\ .\]
\end{lemma}
\begin{proof} By definition of $E_{[I]}$, it is the unique element of $\mathcal{T}_{d,n}$ satisfying, for all complex characters $\chi$ of $\mathcal{T}_{d,n}$,
\begin{equation}\label{eq-lem1}
\chi(E_{[I]})=1\ \ \ \ \text{if $I(\chi)\in[I]$},\ \ \ \ \quad\text{and}\ \ \ \ \quad \chi(E_{[I]})=0\ \ \ \text{otherwise\,,}
\end{equation}
where $I(\chi)$ is the ordered partition in $\cP_d(n)$ corresponding to $\chi$. The condition $I(\chi)\in[I]$ means that $I(\chi)=I^{\sigma}$ for some $\sigma\in\mathfrak{S}_d$. From the definition of $I(\chi)$, this is equivalent to the property that,
$$\text{for $i,j=1,\dots,n\,$:}\ \ \qquad \pos_i(I)=\pos_j(I)\ \ \ \ \Longleftrightarrow\ \ \ \ \chi(t_i)=\chi(t_j)\ .$$
On the other hand, we have, for any $i,j\in\{1,\dots,n\}$ and any character $\chi$, that $\chi(e_{i,j})=1$ if $\chi(t_i)=\chi(t_j)$ and $\chi(e_{i,j})=0$ otherwise. Therefore, the element $e_{[I]}$ satisfies the same formulas (\ref{eq-lem1}) as the element $E_{[I]}$ and thus coincides with it.
\end{proof}

\begin{Prop}\label{prop-inv}
The subalgebra $Y_{d,n}^{\mathfrak{S}_d}$ of $Y_{d,n}$ coincides with the subalgebra generated by $g_1,\dots,g_{n-1}$ and $e_1,\dots,e_{n-1}$.
\end{Prop}
\begin{proof}
From Formulas (\ref{action-gen}) giving the action of $\mathfrak{S}_d$ on $Y_{d,n}$, we have immediately that $g_1,\dots,g_{n-1}\in Y_{d,n}^{\mathfrak{S}_d}$.
The fact that $e_1,\dots,e_{n-1}\in Y_{d,n}^{\mathfrak{S}_d}$ follows directly from Formula (\ref{form-e2}) with the remark that the property $\pos_i(I)=\pos_{i+1}(I)$ is unaltered by the action of $\mathfrak{S}_d$ on $I$.

Reciprocally, Lemma \ref{lem-E-e} shows that any element of the basis (\ref{basis-inv}) of $Y_{d,n}^{\mathfrak{S}_d}$ is in the subalgebra generated by $g_1,\dots,g_{n-1}$ and $e_1,\dots,e_{n-1}$.
\end{proof}

\begin{Rem}\label{image-braid}$\ $
\begin{itemize}
\item[\textbf{(i)}] Due to Proposition \ref{prop-inv}, the algebra $Y_{d,n}^{\mathfrak{S}_d}$ of fixed points under the action of the symmetric group $\mathfrak{S}_d$ has the following interpretation and motivation in connections with knot theory. The generators $g_i$'s are invertible and satisfy the braid relations, thus providing a morphism from the algebra of the braid group on $n$ strands into $Y_{d,n}$. By definition, the image of this morphism is the subalgebra generated by elements $g_i^{\pm1}$'s. 

Let extend the ring $\C[q,q^{-1}]$ by localizing it with respect to $q-q^{-1}$. Over this extended ring, it follows from Proposition \ref{prop-inv} and the formula $g_i^{-1}=g_i-(q-q^{-1})e_i$ that $Y_{d,n}^{\mathfrak{S}_d}$ coincides with the subalgebra generated by the $g_i^{\pm1}$'s, and therefore coincides with the image of the algebra of the braid group inside $Y_{d,n}$.

\item[\textbf{(ii)}] Due to the characteristic equation (\ref{char-g}), the image of the algebra of the braid group (over $\C[q,q^{-1}]$) inside $Y_{d,n}$ can also be described as the subalgebra generated only by the generators $g_i$'s.
\end{itemize}
\end{Rem}

\subsection{The algebra of braids and ties}

The algebra $BT_n$ is the associative $\C[q,q^{-1}]$-algebra with generators
$$\tg_1,\ldots,\tg_{n-1},\te_1,\ldots,\te_{n-1},$$
subject to the defining relations:
\begin{equation}\label{def-3a}
\left\{\begin{array}{rclcccl}
\tg_i\tg_j & = & \tg_j\tg_i &&&\qquad & \text{for $i,j\in\{1,\ldots,n-1\}$ such that $\vert i-j\vert > 1$\,,}\\[0.2em]
\tg_i\tg_{i+1}\tg_i & = & \tg_{i+1}\tg_i\tg_{i+1} &&&& \text{for $i\in\{1,\ldots,n-2\}$\,,}\\[0.2em]
\tg_i^2  & = & 1 + (q-q^{-1}) \, \te_{i} \, \tg_i &&&\qquad& \text{for $i\in\{1,\ldots,n-1\}$\,,}
\end{array}
\right.
\end{equation}
\begin{equation}\label{def-3b}
\left\{\hspace{0.0cm}\begin{array}{rclcccl}
\te_i^2 & = & \te_i &&&& \text{for $i\in\{1,\dots,n-1\}$\,,}\\[0.2em]
\te_i\te_j & =  & \te_j\te_i &&&\qquad\ \ \,&  \text{for $i,j\in\{1,\ldots,n-1\}$\,,}\\[0.2em]
\tg_i\te_j & = & \te_j\tg_i &&&& \text{for $i,j\in\{1,\ldots,n-1\}$ such that $|i-j|\neq 1$\,,}\\[0.2em]
\te_i\tg_j\tg_i   & =  &  \tg_j\tg_i\te_j & & && \text{for $i,j\in\{1,\ldots,n-1\}$ such that $|i-j|=1$\,,}\\[0.2em]
\te_i\te_j\tg_j   & =  &  \te_i\tg_j\te_i & = & \tg_j\te_i\te_j & & \text{for $i,j\in\{1,\ldots,n-1\}$ such that $|i-j|=1$\,.}
\end{array}
\right.\end{equation}

It is a straightforward verification to check that all these defining relations are satisfied in $Y_{d,n}$, for any $d>0$, if we replace elements $\tg_i$ by $g_i$ and elements $\te_i$ by $e_i$. In other words, there is a morphism of algebras from $BT_n$ to $Y_{d,n}$ given by
\begin{equation}\label{morp-BT}
\phi_{d,n}\ \ :\ \ \tg_i\mapsto g_i\ \ \ \ \text{and}\ \ \ \ \te_i\mapsto e_i\,,\ \ \ \ \ \ \ \text{for $i=1,\dots,n-1$.}
\end{equation}

For $i=1,\dots,n-1$, let $\te_{i,i}:=1$. For $i,j\in\{1,\dots,n\}$, if $i<j$ then let
\[\te_{i,j}:=\tg_{j-1}\dots \tg_{i+1}\cdot\te_i\cdot\tg_{i+1}^{-1}\dots\tg_{j-1}^{-1}\ ,\]
so that we have in particular $\te_{i,i+1}=\te_i$ for $i=1,\dots,n-1$. Finally if $i>j$ let $\te_{i,j}:=\te_{j,i}$.

\begin{Rem}
For any $i,j=1,\dots,n-1$, we have $\phi_{d,n}(\te_{i,j})=e_{i,j}$, with $e_{i,j}$ defined in (\ref{def-eij}).
\end{Rem}

Consider an arbitrary (unordered) partition of the set $\{1,\dots,n\}$. Of course it has at most $n$ non-empty parts so we can see it as a class $[I]\in\cP_n(n)/\mathfrak{S}_n$ for a certain element $I$ in $\cP_n(n)$. We set
\begin{equation}\label{def-tfI}
\tilde{f}_{[I]}:=\prod_{\text{\scriptsize{$\begin{array}{c}i,j=1,\dots,n \\
\pos_i(I)=\pos_{j}(I)
\end{array}$}}}\hspace{-0.4cm}\te_{i,j}\ ,
\end{equation}
which is well-defined for similar reasons as elements $e_{[I]}$ in (\ref{def-eI}).

In \cite{Ry}, it is proved by direct calculations that the following set is a spanning set of $BT_n$:
\begin{equation}\label{basis-BT}
\{\ \tilde{f}_{[I]}\tg_w\ |\ [I]\in\cP_n(n)/\mathfrak{S}_n\,,\ w\in\mathfrak{S}_n\ \}\ ,
\end{equation}
where $\tg_w$, for $w\in\mathfrak{S}_n$, is defined as in $Y_{d,n}$ since the braid relations are satisfied in $BT_n$.

Using the results on $Y_{d,n}^{\mathfrak{S}_d}$ from the previous subsection, we obtain the following corollary on the structure of $BT_n$.
\begin{Cor}\label{coro-BT}
\begin{itemize}
\item[\textbf{(i)}] The set (\ref{basis-BT}) is a $\C[q,q^{-1}]$-basis of $BT_n$, which is therefore free of rank $B_n(n)\,n!$\ .
\item[\textbf{(ii)}] The algebra $BT_n$ is isomorphic to the subalgebra of $Y_{d,n}$ generated by $g_1,\dots,g_{n-1},e_1,\dots,e_{n-1}$ if and only if $d\geq n$.
\end{itemize}
\end{Cor}
\begin{proof}
First note that for any $d>0$, the image of the morphism $\phi_{d,n}$ defined in (\ref{morp-BT}) is the subalgebra of $Y_{d,n}$ generated by $g_1,\dots,g_{n-1},e_1,\dots,e_{n-1}$, which is, by Proposition \ref{prop-inv}, isomorphic to $Y_{d,n}^{\mathfrak{S}_d}$. We know that $Y_{d,n}^{\mathfrak{S}_d}$ is free of rank $B_d(n)\,n!\ $.

Assume that $d\geq n$. Then, in this case, we have $B_d(n)=B_n(n)$. Therefore the image of the set (\ref{basis-BT}) by $\phi_{d,n}$ form a spanning set of $Y_{d,n}^{\mathfrak{S}_d}$ of cardinality the rank of $Y_{d,n}^{\mathfrak{S}_d}$. We conclude that the set (\ref{basis-BT}) is linearly independent in $BT_n$, which shows \textbf{(i)}. We also conclude that, when $d\geq n$, the morphism $\phi_{d,n}$ provides an isomorphism between $BT_n$ and the subalgebra of $Y_{d,n}$ generated by $g_1,\dots,g_{n-1},e_1,\dots,e_{n-1}$.

Now let $d<n$. Then in this case $B_d(n)<B_n(n)$ (since, the partition $\{1\}\cup\dots\cup\{n\}$ is not counted in $B_d(n)$). We already know that the rank of $BT_n$ is $B_n(n)\,n!$ while the rank of  the subalgebra of $Y_{d,n}$ generated by $g_1,\dots,g_{n-1},e_1,\dots,e_{n-1}$ is $B_d(n)\,n!$. We conclude that they cannot be isomorphic.
\end{proof}

\begin{Rem}
Item $\textbf{(i)}$  was proved in \cite{Ry} by a different argument relying on the construction of a certain faithful representation of $BT_n$ on a tensor space. The ``if'' part of $\textbf{(ii)}$ was subsequently obtained in \cite{ER}.
\end{Rem}

\subsection{Parametrization of the simple $\mathbb{C}_{\theta} Y_{d,n}^{\mathfrak{S}_d}$-modules}
 
 Let us now apply the results of Subsection \ref{clifford} in order to study the representation theory of  $\mathbb{C}_{\theta} Y_{d,n}^{\mathfrak{S}_d}$ where  $\theta : \mathbb{C}[q,q^{-1}] \to \mathbb{C}$ is a specialization.  We first need to understand precisely the structure of the inertia subgroups. 
 Let $\ulambda \in  \Lambda^e_{\mu}$ with $\mu\in \operatorname{Comp}_d (n)$.  
  There exists $k:=k(\ulambda) \in\{1,\ldots,d\}$ and 
 $1\leq i_1 < \ldots <i_k \leq d$  such that 
 $\lambda^{i_1}\neq \ldots\neq  \lambda^{i_k}$ and  such that 
$$\{\lambda^{i_1},\ldots, \lambda^{i_k}\}=\{\lambda^1,\ldots,\lambda^d\}.$$
We set for all $s=1,\ldots,k$ :
$$J_{s}=\{ j\in \{1,\ldots,d\}\ |\ \lambda^j=\lambda^{i_s}\},$$
so that we have 
$$\{1,\ldots, n\}=J_1\sqcup \ldots \sqcup J_k.$$
Then the inertia subgroup of $V^{\ulambda}$ as defined by \cite{RR} is 
$$\mathfrak{H}_{V^\ulambda} (\mathfrak{S}_d)=\{\sigma\in \mathfrak{S}_d\ |\ \forall s\in \{1,\ldots,k\},  \sigma (J_s)=J_s \}.$$
This is a Young subgroup of $\mathfrak{S}_d$ conjugate to 
$$\mathfrak{S}_{|J_1|}\times \ldots \times 
 \mathfrak{S}_{|J_k|}.$$
The cardinal $|J_1|!\ldots |J_k|!$ of $\mathfrak{H}_{V^\ulambda} (\mathfrak{S}_d)$  
 is denoted by $x(\ulambda)\,$.
 
 \begin{Rem}
If we have $\lambda^i\neq \lambda^j$ for all $i\neq j$, we obtain
 $\mathfrak{H}_{V^\ulambda} (\mathfrak{S}_d)=\{\textrm{Id}\}$ and if $\lambda^1=\ldots=\lambda^d$, we 
  obtain   $\mathfrak{H}_{V^\ulambda} (\mathfrak{S}_d)=\mathfrak{S}_d$. 
  \end{Rem}
 In the general case, note that the simple $\mathfrak{H}_{V^\ulambda} (\mathfrak{S}_d)$-modules 
  are parametrized by the set 
 $$\mathcal{M}(\ulambda):=\{\unu=(\nu^1,\ldots,\nu^k),\ \forall i\in \{1,\ldots, k\},\ \nu^i\in \Lambda_{ |J_i|}\}$$  
 and we denote 
 $$\textrm{Irr} (\mathfrak{H}_{V^\ulambda} (\mathfrak{S}_d))=\{S^{\unu}\ |\ \unu\in \mathcal{M} (\ulambda)\}$$ 
We can then apply Proposition \ref{Cliff} to deduce the following result which provides a parametrization 
 of the simple $\C_{\theta}Y_{d,n}^{\mathfrak{S}_d} $-modules. Such a parametrization has been also obtained in \cite{Ry}  in the semisimple case (that is for $e>>0$ or $e=\infty$).

\begin{Th} Keeping the above notations, 
the simple $\C_{\theta}Y_{d,n}^{\mathfrak{S}_d}$-modules are given by the set 
$$\{ V^{\ulambda}_{S^\unu}\ |\ \ulambda \in \Lambda_{\mu}^e, \mu\in \operatorname{Comp}_d (n), \unu \in \mathcal{M} (\ulambda)\}$$
and we have $V^{\ulambda}_{S^{\unu}}\simeq V^{\ulambda'}_{S^{\unu'}}$ if and only if 
     $\ulambda$ and $\ulambda'$ are in the same $\mathfrak{S}_d$-orbit and $\unu=\unu'$. In addition we have
     $$\operatorname{dim}_{\mathbb{C}} (V^{\ulambda}_{S^{\unu}}) =\frac{y (\ulambda,e)(\prod_{1\leq i\leq k} y (\nu^i,\infty) )}{x(\ulambda)}$$
\end{Th}

\section{Action of the cyclic group $\mathbb{Z}/p\mathbb{Z}$}

In this section, we consider the cyclic group $\mathbb{Z}/p\mathbb{Z}$ (for $p$ dividing $d$) as a subgroup of $\mathfrak{S}_d$, and we study its action on the Yokonuma--Hecke algebra. We will see that the subalgebra of fixed points can be seen as a new deformation of the complex reflection group of type $G(d,p,n)$. 

\subsection{Subalgebra $Y_{d,n}^{\mathbb{Z}/p\mathbb{Z}}$}\label{subsec-51}

Let us consider the element $\sigma_{d/p}$ of $\mathfrak{S}_d$ such that 
for all $a\in \{1,\ldots,d\}$ we have $\sigma_{d/p} (a)=a+d/p (\textrm{mod }d)$.

Let $I\in\cP_d(n)$. By construction of the idempotent $E_I$, we have $t_iE_I=\xi_{\pos_i(I)}E_I$ for any $i\in\{1,\dots,n\}$. We recall that the $d$-th roots of unity $\{\xi_1,\dots,\xi_d\}$ are ordered such that $\xi_1:=\textrm{exp}(2i\pi/d)$ and $\xi_a=\xi_1^a$ for $a=1,\ldots,d$. Using Formula (\ref{form-pos}), we obtain 
\[t_iE_{I^{\sigma_{d/p}}}=\xi_{\pos_i(I)+d/p(\textrm{mod }d)}E_{I^{\sigma_{d/p}}}=\xi_1^{d/p}t_iE_{I^{\sigma_{d/p}}}\ .\]
Since this is true for any $I\in\cP_d(n)$, we have that, for all $i=1,\ldots,n$, the action of $\sigma_{d/p}$ on the generators $t_i$ is given by:
\begin{equation}\label{act-ti}
t_i^{\sigma_{d/p}}=\xi_1^{d/p} t_i\ ,
\end{equation}
 where $\xi_1^{d/p}=\textrm{exp}(2i\pi/p)$ is thus a primitive root of order $p$. Note that the action 
 on $t_i$ of an arbitrary element of $\mathfrak{S}_d$ is in general far more complicated to describe. 
We consider the action on $Y_{d,n}$ of the subgroup generated by $\sigma_{d/p}$, which is isomorphic to $\mathbb{Z}/p\mathbb{Z}$. We denote $Y_{d,n}^{\mathbb{Z}/p\mathbb{Z}}$ the corresponding subalgebra of fixed points.

We will use the following notation. For $i=1,\ldots,n-1$, we set 
$$a_{i}:=t_{i}^{-1}t_{i+1}\ ,$$
and we set $a_0=t_1^p$. Note that we have $a_0^{d/p}=1$ and $a_i^{d}=1$ for all $i=1,\ldots,n-1$.
The following proposition gives a basis of $Y_{d,n}^{\mathbb{Z}/p\mathbb{Z}}$. 
\begin{Prop}\label{base}
 A  basis of $Y_{d,n}^{\mathbb{Z}/p\mathbb{Z}}$ is given by the following set
\begin{equation}\label{basisfixed-t}
\mathfrak{B}_1:=\{\ t_1^{\alpha_1}\dots t_n^{\alpha_n}g_w\ |\ \alpha_1,\dots,\alpha_n\in\{0,\dots,d-1\}\,,\ \alpha_1+\ldots+\alpha_n\equiv 0 (\operatorname{mod }p)\,,\ w\in\mathfrak{S}_n\ \}\ .
\end{equation}
In addition, the elements of this basis can be rewritten as follows: 
\begin{equation}\label{basisfixed-a}
\{\ a_0^{\beta_0}\dots a_{n-1}^{\beta_{n-1}}g_w\ |\ \beta_1,\dots,\beta_{n-1}\in\{0,\dots,d-1\}\,,\ \beta_0\in \{0,\ldots,d/p-1\}\,,\ w\in\mathfrak{S}_n \}\ .
\end{equation}
In particular we have 
$$\operatorname{rk} (Y_{d,n}^{\mathbb{Z}/p\mathbb{Z}})=d^n n!/p\ .$$
\end{Prop}
\begin{proof}
Consider the standard basis (\ref{basis-t}) of $Y_{d,n}$. Each of these elements is multiplied by a power of $\xi_1^{d/p}$ under the action of $\sigma_{d/p}$. It is easy to see that the elements in $\mathfrak{B}_1$ are precisely the ones which are multiplied by 1. Therefore, for an arbitrary element $h$ of $Y_{d,n}$ written as a linear combination of the standard basis elements, we have $h\in Y_{d,n}^{\mathbb{Z}/p\mathbb{Z}}$ if and only if $h$ is a linear combination of the elements in $\mathfrak{B}_1$. This shows that $\mathfrak{B}_1$ is a generating set of $Y_{d,n}^{\mathbb{Z}/p\mathbb{Z}}$. The linear independence of the elements of $\mathfrak{B}_1$ follows from the fact that $\mathfrak{B}_1$ is a subset of the standard basis of $Y_{d,n}$.

Now let us consider an element  $y=t_1^{\alpha_1}\dots t_n^{\alpha_n}$
   with $\alpha_1,\dots,\alpha_n\in\{1,\dots,d\}$
  and  $\alpha_1+\ldots+\alpha_n\equiv 0 (\textrm{mod }p)$. We have:
  $$\begin{array}{rcl}
  y&=& t_1^{\alpha_1}\ldots t_{n-1}^{\alpha_{n-1}+\alpha_n} a_{n-1}^{\alpha_n}\\
  &=&\ldots \\
  &=&t_1^{\sum_{1\leq i\leq n} \alpha_i} a_{1}^{\sum_{2\leq i\leq n} \alpha_i}\ldots a_{n-2}^{\alpha_{n-1}+\alpha_n} a_{n-1}^{\alpha_n}\\[0.4em]
  &=&a_0^{(\sum_{1\leq i\leq n} \alpha_i)/p} a_{1}^{\sum_{2\leq i\leq n} \alpha_i}\ldots a_{n-2}^{\alpha_{n-1}+\alpha_n} a_{n-1}^{\alpha_n}\ .
  \end{array}$$
Taking into account that $a_0$ is of order $d/p$ and, for $i=1,\ldots,n-1$,  $a_i$ is of order $d$, this shows that any element of $\mathfrak{B}_1$ is also an element of (\ref{basisfixed-a}). As the two sets  have the same cardinality, they coincide. This concludes the proof. 
\end{proof}

\begin{Prop}
There is a presentation of $Y_{d,n}^{\mathbb{Z}/p\mathbb{Z}}$  by generators
 $g_1, g_2,\ldots g_{n-1}, a_0, a_1,\ldots,a_{n-1}$
 and the following relations:
\begin{itemize}
\item \rm{\textbf{(R1)}}\ :\ \ \ $\left\{\begin{array}{rclcl}
g_ig_j & = & g_jg_i &\quad & \text{for $i,j\in\{1,\ldots,n-1\}$ such that $\vert i-j\vert > 1$\,,}\\[0.2em]
g_ig_{i+1}g_i & = & g_{i+1}g_ig_{i+1} && \text{for $i\in\{1,\ldots,n-2\}$\,,}\\[0.2em]
g_i^2  & = & 1 + (q-q^{-1}) \, e_{i} \, g_i &\quad& \text{for $i\in\{1,\ldots,n-1\}$\,,}
\end{array}\right.$
\item  \rm{\textbf{(R2)}}\ :\ \ \ $\left\{\begin{array}{rclcl}
a_ia_j & = & a_ja_i &\quad & \text{for $i,j\in\{0,\ldots,n-1\}$ ,}\\[0.2em]
a_0^{d/p} & = & 1\,, && \\[0.2em]
a_i^d  & = & 1 &\quad& \text{for $i\in\{1,\ldots,n-1\}$\,,}
\end{array}\right.$
\item  \rm{\textbf{(R3)}}\ :\ \ \ $g_i a_0 g_i^{-1}=\left\{ \begin{array}{ll} a_0 a_1^p\qquad & \text{if $i=1$\,,}\\[0.2em]
a_0 & \text{for $i=2,\dots,n-1$\,,}\end{array}  \right.$
\item  \rm{\textbf{(R4)}}\ :\ \ \ $g_i a_j g_i^{-1}=\left\{ \begin{array}{ll} a_{j-1} a_j \qquad & \text{if $i=j-1$\,,}\\[0.2em]
a_j^{-1} & \text{if $i=j$\,,}\\[0.2em]
a_ja_{j+1} & \text{if $i=j+1$\,,}\\[0.2em]
a_j & \text{otherwise\,,} \end{array} \right.\,,\ \ \ \ \ \text{for $i,j=1,\dots,n-1$\,,}$
\end{itemize}
where $e_i:=\displaystyle \frac{1}{d}\sum_{1\leq s\leq d}a_i^s$\,, \ for $i=1,\dots,n-1$.
\end{Prop}
\begin{proof}
The defining relations (R1)--(R4) are easily checked to be satisfied in $Y_{d,n}$. Let us denote $\mathfrak{Y}_{d,p,n}$ the algebra defined with the above presentation (by a slight abuse of notation, we keep the same notations than for the elements of $Y_{d,n}$). Thus we have an algebra homomorphism:
$$\Theta : \mathfrak{Y}_{d,p,n} \to Y_{d,n}^{\mathbb{Z}/p\mathbb{Z}}$$
which is obviously surjective, using Proposition \ref{base}.

Now, the relations (R1)--(R4) clearly allows to write any element of $\mathfrak{Y}_{d,p,n}$ as a linear combination of elements in (\ref{basisfixed-a}) (seen as a subset of elements of $\mathfrak{Y}_{d,p,n}$). So (\ref{basisfixed-a}) is a generating set for $\mathfrak{Y}_{d,p,n}$. It is also linearly independent since it is in $Y_{d,n}^{\Z/p\Z}$ and because $\Theta$ is a morphism. As a consequence, (\ref{basisfixed-a}) is also a basis of $\mathfrak{Y}_{d,p,n}$ and $\Theta$ is thus bijective. 
\end{proof}

The following lemma, that we will need later, reduces the number of generators and of relations in the above presentation of $Y_{d,n}^{\mathbb{Z}/p\mathbb{Z}}$.
\begin{lemma}\label{lem-pres}
There is a presentation of $Y_{d,n}^{\mathbb{Z}/p\mathbb{Z}}$  by generators
 $g_1, g_2,\ldots g_{n-1}, a_0, a_1$ and defining relations:
\[\text{\rm{\textbf{(R1)}}\,, \ \rm{\textbf{(R3)}}\,, \ \rm{\textbf{(R4)}} for $j=1$\,, \ \ $a_0^{d/p}=1$\,,\ \ \ $a_0a_1=a_1a_0$\ \ \ and\ \ \ $a_1a_2=a_2a_1$\ ,}\]
where the elements $a_i$, $i=2,\dots,n-1$, and $e_i$, $i=1,\dots,n-1$ are defined by:
\[a_i:=g_{i-1}g_i a_{i-1}g^{-1}_{i}g^{-1}_{i-1}\,,\ \ \ i=2,\dots,n-1\,,\ \ \ \ \ \text{and}\ \ \ \ \ \  e_i:=\displaystyle \frac{1}{d}\sum_{1\leq s\leq d}a_i^s\,, \ \ i=1,\dots,n-1\,.\]
\end{lemma}
\begin{proof} It is easy to see that the elements $a_i$ ($i=2,\ldots,n-1$) and the elements 
 $e_i$ ($i=1,\ldots,n-1$) satisfy the relations of the lemma in $Y_{d,n}$.  
As the relations of the lemma are contained in (R1)--(R4), we only have to check that Relations (R1)--(R4) are consequences of the relations listed in the lemma.  Note first that, using (R3), we have that $a_0a_1^p$ is conjugate to $a_0$. As $a_0$ and $a_1$ commutes and $a_0^{d/p}=1$, we obtain that $a_1^d=1$. Since all $a_i$, $i=1,\dots,n-1$, are conjugate, we thus have that $a_i^d=1$ for $i=1,\dots,d-1$.

\vskip .2cm
We now check Relations (R4) by induction on $j$ (for $j=1$ it is listed in the lemma). Let $j>1$. Note that, by induction hypothesis, we have $a_j=a_{j-1}^{-1}g_{j}a_{j-1}g_j^{-1}$. We will treat several cases.
\begin{itemize}
\item Let $i<j-2$ or $i>j+1$. Then $g_i$ commutes with $a_{j-1}$ by induction hypothesis and $g_i$ commutes with $g_{j}$. So $g_i$ commutes with $a_j=a_{j-1}^{-1}g_{j}a_{j-1}g_j^{-1}$.

\item Let $i=j-2$. By induction  hypothesis, we have $g_{j-2}a_{j-1}=a_{j-2}a_{j-1}g_{j-2}$ and, taking inverses, we have also $g_{j-2}a_{j-1}^{-1}=a_{j-1}^{-1}a_{j-2}^{-1}g_{j-2}$. Therefore we have, using that $g_{j-2}$ and $g_j$ commute,
\[g_{j-2}a_{j}=g_{j-2}\cdot a_{j-1}^{-1}\,g_{j}\,a_{j-1}\,g_j^{-1}=a_{j-1}^{-1}a_{j-2}^{-1}\,g_{j}\,a_{j-2}a_{j-1}\,g_j^{-1}\cdot g_{j-2}= a_jg_{j-2}\ ,\]
where we used in the last equality that $g_j$ commutes with $a_{j-2}$ by induction hypothesis.

\item Let $i=j-1$. By induction hypothesis, we have $g_{j-1}a_{j-1}^{-1}=a_{j-1}g_{j-1}$. Therefore, we obtain
\[g_{j-1}\cdot a_j\cdot g_{j-1}^{-1}=g_{j-1}\cdot a_{j-1}^{-1}g_{j}a_{j-1}g_j^{-1}\cdot g_{j-1}^{-1}=a_{j-1}\cdot g_{j-1}g_{j}a_{j-1}g_j^{-1}g_{j-1}^{-1}=a_{j-1}a_j \ ,\]
by definition of $a_j$.

\item Let $i=j$. We have:
\[g_ja_j=g_j\cdot g_{j-1}g_ja_{j-1}g_j^{-1}g_{j-1}^{-1}=g_{j-1}g_ja_{j-1}^{-1}g_j^{-1}g_{j-1}^{-1}\cdot g_j=a_{j}^{-1}g_j\ ,\]
where we used successively the definition of $a_j$, the braid relation $g_j\cdot g_{j-1}g_j=g_{j-1}g_j\cdot g_{j-1}$, the induction hypothesis $g_{j-1}a_{j-1}=a_{j-1}^{-1}g_{j-1}$ and the braid relation $g_{j-1}\cdot g_j^{-1}g_{j-1}^{-1}= g_j^{-1}g_{j-1}^{-1}\cdot g_j$.

\item Let $i=j+1$. By definition, we have $a_{j+1}=g_jg_{j+1}a_jg_{j+1}^{-1}g_j^{-1}$. We replace $a_j$ by $a_{j-1}^{-1}g_{j}a_{j-1}g_j^{-1}$ and we find
\[a_{j+1}=g_jg_{j+1}\cdot a_{j-1}^{-1}g_{j}a_{j-1}g_j^{-1}\cdot g_{j+1}^{-1}g_j^{-1}=g_ja_{j-1}^{-1}\cdot g_{j+1}g_jg_{j+1}^{-1}\cdot a_{j-1}g_j^{-1}g_{j+1}^{-1}\ ,\]
where we used the braid relation $g_j^{-1}g_{j+1}^{-1}g_j^{-1}=g_{j+1}^{-1}g_j^{-1}g_{j+1}^{-1}$ and that $g_{j+1}$ commutes with $a_{j-1}$ by induction hypothesis. Then we transform the expression $g_{j+1}g_jg_{j+1}^{-1}$ into $g_j^{-1}g_{j+1}g_j$ and we use that $g_ja_{j-1}g_j^{-1}=a_{j-1}a_j$ (and its inverse) by induction hypothesis. We obtain
\[a_{j+1}=a_j^{-1}a_{j-1}^{-1}\cdot g_{j+1}\cdot a_{j-1}a_j\cdot g_{j+1}^{-1}=a_j^{-1}g_{j+1}a_jg_{j+1}^{-1}\ ,\]
where we used in the last equality that $g_{j+1}$ commutes with $a_{j-1}$ by induction hypothesis. We conclude that $g_{j+1}a_jg_{j+1}^{-1}=a_ja_{j+1}$.
\end{itemize}

It remains to check that $a_ia_j=a_ja_i$ for any $i,j\in\{0,1\dots,n-1\}$ so that (R2) will be satisfied. It is enough to consider $j>i$. First let $i=0$. We use induction on $j$. For $j=1$, $a_0a_1=a_1a_0$ is listed in the lemma. Let $j>1$. From Relation (R4), we have $a_j=a_{j-1}^{-1}g_ja_{j-1}g_j^{-1}$. Since $a_0$ commutes with $g_j$ by (R3) and $a_0$ commutes with $a_{j-1}$ by induction hypothesis, we obtain that $a_0$ commutes with $a_j$.

Let now $i>0$. We treat several cases.
\begin{itemize}
\item Let $i=1$. We use induction on $j$. For $j=2$, $a_1a_2=a_2a_1$ is listed in the lemma. Let $j>2$. We write $a_j=a_{j-1}^{-1}g_ja_{j-1}g_j^{-1}$, from Relation (R4), and then we use that $a_1$ commutes with $g_j$, from (R4), and with $a_{j-1}$ by induction hypothesis. So $a_1$ commutes with $a_j$.

\item Let now $i>1$ and $j>i+1$. We use induction on $i$. We write $a_i=a_{i-1}^{-1}g_ia_{i-1}g_i^{-1}$. By (R4), $a_j$ commutes with $g_i$ and $a_j$ commutes with $a_{i-1}$ by induction hypothesis. So $a_i$ and $a_j$ commute.

\item It remains to consider $i>1$ and $j=i+1$. We still use induction on $i$. Since, from (R4), $a_i=a_{i-1}^{-1}g_ia_{i-1}g_i^{-1}$ and $a_{i+1}$ commutes with $a_{i-1}$ by induction hypothesis, we have that $a_{i+1}$ commutes with $a_i$ if and only if $a_{i+1}$ commutes with $g_ia_{i-1}g_i^{-1}$. Recall that, by (R4), we have $g_ia_{i+1}g_i^{-1}=a_ia_{i+1}$. This easily implies the following:
\[a_{i+1}g_i^{-1}=g_i^ {-1}a_ia_{i+1}\ \ \ \ \ \ \text{and}\ \ \ \ \ \ \ a_{i+1}g_i=a_i^{-1}g_ia_{i+1}\ .\]
We use these two relations together with the fact that $a_{i+1}$ and $a_{i-1}$ commute to calculate:
\[a_{i+1}\cdot g_ia_{i-1}g_i^{-1}=a_i^{-1}g_ia_{i+1}\cdot a_{i-1}g_i^{-1}=a_i^{-1}g_ia_{i-1}\cdot a_{i+1} g_i^{-1}=a_i^{-1}g_ia_{i-1}g_i^{-1}a_i\cdot a_{i+1}\]
It remains to note that $a_i$ commutes with $g_ia_{i-1}g_i^{-1}$. This is equivalent to say that $a_i$ commutes with $a_{i-1}$, which is known by induction hypothesis. This concludes the verification that $a_{i+1}$ commutes with $a_i$ and the proof of the lemma.
\end{itemize}
\end{proof}

\begin{Rem}\label{rem-lem} $\ $

\begin{itemize}

\item[\textbf{(i)}] Consider the group generated by $g_1,\dots,g_{n-1},a_0,a_1,\dots,a_{n-1}$ with all relations (R1)--(R4) except the quadratic relation for the $g_i$'s and $a_0^{d/p}=a_i^d=1$. Then it is easy to see, looking at the proof, that the statement of the lemma (with the quadratic relation for the $g_i$'s and $a_0^{d/p}=1$ removed) is still valid for this group.

\item[\textbf{(ii)}] We apply the preceding remark to the case $p=1$. In this case, the group is the framed braid group $B_n\ltimes\Z^n$, where $B_n$ is the usual braid group. Denote in this case $a_0=t_1$ to be consistent. With (R3), element $a_1$ is now expressed as $a_1=t_1^{-1}g_1t_1g_1^{-1}$, and thus can be removed from the set of generators. The lemma gives then a presentation of this group only in terms of $g_1,\dots,g_{n-1}$ and $t_1$. It is straightforward to see that the obtained presentation is equivalent to the following: the group $B_n\ltimes\Z^n$ is generated by $g_1,\dots,g_{n-1}$ and $t_1$ with defining relations:
\[\text{\rm{\textbf{(Bra1)}}}\,,\ \ \ g_1t_1g_1^{-1}=g_1^{-1}t_1g_1\,,\ \ \ g_1t_1g_1^{-1}t_1=t_1g_1t_1g_1^{-1}\ \ \ \ \text{and}\ \ \ g_it_1=t_1g_i\ (i=2,\dots,n-1)\,,\]
where $\rm{\textbf{(Bra1)}}$ consists of the two first lines in $\rm{\textbf{(R1)}}$.

\item[\textbf{(iii)}] Pursuing further the preceding remark, we obtain a presentation of the Yokonuma--Hecke algebra only in terms of $g_1,\dots,g_{n-1}$ and $t_1$. Indeed, if we define $t_2,\dots,t_{n-1}$ by $t_{i+1}:=g_it_ig_i^{-1}$, then $Y_{d,n}$ is the quotient of the group algebra of $B_n\ltimes\Z^n$ (with the presentation above) by the relations $t_1^d=1$ and $g_i^2=1+(q-q^{-1})e_ig_i\,$, where $e_i:=\frac{1}{d}\sum_{1\leq s\leq d}t_i^st_{i+1}^{-s}$. 
We indicate that, in $Y_{d,n}$, the relation $g_1t_1g_1^{-1}=g_1^{-1}t_1g_1$ is not necessary as a defining relation, since it is actually a consequence of the others.

Consider the specializations $q\mapsto\pm1$ of $Y_{d,n}$. Denoting $s_i$ the specialization of $g_i$, the generators are now $t_1,s_1,\dots,s_{n-1}$ and the relations become:
\[\text{\rm{\textbf{(Brb1)}}}\,,\ \ \ s_i^2=1\ (i=1,\dots,n-1)\,,\ \ \ t_1^d=1\,,\ \ \ s_1t_1s_1t_1=t_1s_1t_1s_1\ \ \ \ \text{and}\ \ \ s_it_1=t_1s_i\ (i=2,\dots,n-1)\,.\]
This recovers the standard presentation of the complex reflection group $G(d,1,n)$ and the well-known fact that $Y_{d,n}$ is a deformation of $G(d,1,n)$.

\end{itemize}

\end{Rem}

\subsection{Deformations of complex reflection groups of type $G(d,p,n)$}
 
In the classification of irreducible complex reflection groups by Shephard--Todd, there is a collection of $34$ exceptional cases, together with a unique infinite series $G(d,p,n)$, where $d,p,n\geq1$ with $p$ dividing $d$. We refer to \cite{Br} for details.  The braid group of type $G(d,p,n)$ has a presentation by generators $s_0$, $s_1$, \ldots, $s_{n-1}$ and $\alpha_0$ subject to the following relations :
\begin{itemize}
\item \rm{\textbf{(Br1)}}\ : \ \ \ $\left\{\begin{array}{rclcl}
s_is_j & = & s_js_i &\quad & \text{for $i,j\in\{1,\ldots,n-1\}$ such that $\vert i-j\vert > 1$\,,}\\[0.2em]
s_is_{i+1}s_i & = & s_{i+1}s_is_{i+1} && \text{for $i\in\{1,\ldots,n-2\}$\,,}
\end{array}\right.$

\item \rm{\textbf{(Br2)}}\ : \ \ \   $\left\{\begin{array}{rclcl}
s_is_0 & = & s_0s_i &\quad & \text{for $i\in\{3,\ldots,n-1\}$\,,}\\[0.2em]
s_0s_{2}s_0 & = & s_{2}s_0s_{2}\,, &&
\end{array}\right.$

\item \rm{\textbf{(Br3)}}\ : \ \ \  $(s_2 s_0 s_1)^2=(s_0 s_1 s_2)^2$\ ,

\item \rm{\textbf{(Br4)}}\ : \ \ \  $\underbrace{s_1 \alpha_0 s_0 s_1 s_0 s_1 \ldots}_{p+1 \text{ terms}} = \underbrace{\alpha_0 s_0 s_1 s_0 s_1 \ldots}_{p+1 \text{ terms}}$\ ,

\item \rm{\textbf{(Br5)}}\ : \ \ \   $\left\{\begin{array}{rclcl}
\alpha_0 & = & 1\,, &\quad & \text{if $p=d$\,,}\\[0.2em]
\alpha_0 s_0 s_1 & = & s_0 s_1 \alpha_0\,, &&\\[0.2em]
\alpha_0 s_i  & = & s_i \alpha_0 && \text{for $i\in\{2,\ldots,n-1\}$\,.}
\end{array}\right.$
\end{itemize}
The complex reflection group $G(d,p,n)$ is then the quotient of the braid group of type $G(d,p,n)$ by the relations:
\begin{equation}\label{quot-gr}
\alpha_0^{d/p}=1\ \ \ \ \ \ \ \text{and}\ \ \ \ \ \ s_i^2=1\,,\ \ i=0,1,\dots,n-1\,.
\end{equation}

\begin{Rem}\label{rem-braid}$\ $

\begin{itemize}
\item[\textbf{(i)}] Let $p=1$. In this case, (Br4) gives that $s_0=\alpha_0^{-1}s_1\alpha_0$ and therefore, $s_0$ can be removed from the set of generators. It is easy to check that the second line of (Br5) becomes then $s_1\alpha_0s_1\alpha_0=\alpha_0s_1\alpha_0 s_1$, and moreover, Relations (Br2) and (Br3) are consequences of the other defining relations. Therefore, if $d=1$, the braid group of type $G(1,1,n)$ is generated by $s_1,\dots,s_{n-1}$ with Relations (Br1), and is then simply the usual braid group (of type A). If $d>1$, the braid group of type $G(d,1,n)$ is generated by $\alpha_0,s_1,\dots,s_{n-1}$ with relations (Br1) together with
\[s_1\alpha_0s_1\alpha_0=\alpha_0s_1\alpha_0 s_1\ \ \ \ \ \text{and}\ \ \ \ \ \ \alpha_0s_i=s_i\alpha_0\,,\ \ i=2,\dots,n-1\,,\]
and is then simply the braid group of type $B$ (also called affine braid group).

\item[\textbf{(ii)}] Let $p=d$. In this case, $\alpha_0$ is equal to $1$ and thus $\alpha_0$ can be removed from the set of generators. Relations (Br5) disappear and Relation (Br4) becomes $\underbrace{s_1s_0 s_1 \ldots}_{p \text{ terms}} = \underbrace{s_0 s_1 s_0\ldots}_{p\text{ terms}}$.

Assume moreover that $p=d=2$. Then (Br4) is simply $s_0s_1=s_1s_0$ and it is easy to check and well-known that in this case (Br3) is a consequence of the other defining relations. Therefore  the braid group of type $G(2,2,n)$  is simply the braid group of type $D_n$.
 \end{itemize}
 \end{Rem}

\paragraph{Quotients of  braid groups of type $G(d,p,n)$.} We now formulate  the main result of this section, showing in particular that the subalgebra of fixed points $Y_{d,n}^{\Z/p\Z}$ is a deformation of the group $G(d,p,n)$, and moreover expressing it as a quotient of the braid group of type $G(d,p,n)$.

\begin{Th}\label{theo-def}
The algebra $Y_{d,n}^{\Z/p\Z}$ is isomorphic to the quotient of the group algebra of the braid group of type $G(d,p,n)$ by the relations
\begin{equation}\label{quot-Y}
(s_0s_1^{-1})^d=1\,,\ \ \ \ \ \alpha_0^{d/p}=1\ \ \ \ \ \ \text{and}\ \ \ \ \ \ s_i^2=1+(q-q^{-1})e_is_i\ \ (i=0,1,\dots,n-1)\,,
\end{equation}
where $e_0=e_1:=\displaystyle\frac{1}{d}\sum_{1\leq s\leq d}(s_0s_1^{-1})^s$ and $e_i:=s_{i-1}s_ie_{i-1}s_i^{-1}s_{i-1}^{-1}$\,, for $i=2,\dots,n-1$.
\end{Th}
The rest of this section will be devoted to the proof of the theorem. Before that, we give some remarks and examples, focusing on $G(d,1,n)$ in item (i) and on the Weyl group of type $D_n$ (that is, $G(2,2,n)$) in item (ii). 
\begin{Rem}$\ $

\begin{itemize}
\item[\textbf{(i)}] Let $p=1$. We pursue further Remark \ref{rem-braid}(i). Note first that if $d=1$ then $s_0=s_1$, $e_i=1$ and the presentation of the theorem is nothing but the usual presentation of $Y_{1,n}$ (which is the Hecke algebra of type A).

Let $d>1$. The theorem gives then a presentation of the Yokonuma--Hecke algebra $Y_{d,n}$ as a quotient of the group algebra of the braid group of type $B$. For convenience, let us rename in this remark $\alpha_0=t_1$. Recall that here $s_0$ is removed from the set of generators since $s_0=t_1^{-1}s_1t_1$. Then in (\ref{quot-Y}) the relation expressing $s_0^2$ reads $(t_1^{-1}s_1t_1)^2=1+(q-q^{-1})e_1t_1^{-1}s_1t_1$. This is equivalent to $t_1e_1=e_1t_1$ which is equivalent to $t_1s_1^2=s_1^2t_1$, which is in turn equivalent to $s_1^{-1}t_1s_1=s_1t_1s_1^{-1}$. 

Moreover, taking into account the braid relation $t_1s_1t_1s_1=s_1t_1s_1t_1$, this easily implies that $t_1s_1t_1s_1^{-1}=s_1t_1s_1^{-1}t_1$, and in turn that $(t_1^{-1}s_1t_1s_1^{-1})^d=1$ using $t_1^d=1$.

We conclude that, in this case, an equivalent formulation of the theorem is that the Yokonuma--Hecke algebra $Y_{d,n}$ is isomorphic to the quotient of the group algebra of the braid group of type $B$ by the relations:
\begin{equation}\label{quot-YH}
s_1^{-1}t_1s_1=s_1t_1s_1^{-1}\,,\ \ \ \ \ t_1^d=1\ \ \ \ \ \ \text{and}\ \ \ \ \ \ s_i^2=1+(q-q^{-1})e_is_i\ \ (i=1,\dots,n-1)\,,
\end{equation}
where $e_1:=\displaystyle\frac{1}{d}\sum_{1\leq s\leq d}(t_1^{-1}s_1t_1s_1^{-1})^s$ and $e_i:=s_{i-1}s_ie_{i-1}s_i^{-1}s_{i-1}^{-1}$\,, for $i=2,\dots,n-1$.

The correspondence with the usual presentation of $Y_{d,n}$ is simply $t_1\leftrightarrow t_1$ and $s_i \leftrightarrow g_i$ for $i=1,\dots,n-1$. The proof of the general result (to be given below) is much simpler in this case. A sketch is as follows. On one hand, note that all relations of the braid group of type B and relations (\ref{quot-YH}) are satisfied in $Y_{d,n}$. On the other hand, we first reduced the number of generators and relations in the presentation of $Y_{d,n}$: this was done in Lemma \ref{lem-pres} in general and made explicit for $p=1$ in Remark \ref{rem-lem}(iii). So it remains to check that the relation  $t_1s_1t_1s_1^{-1}=s_1t_1s_1^{-1}t_1$ is satisfied in the quotient by (\ref{quot-YH}). This is obvious using $s_1^{-1}t_1s_1=s_1t_1s_1^{-1}$ and $t_1s_1t_1s_1=s_1t_1s_1t_1$.

\item[\textbf{(ii)}] Let $p=d=2$. We pursue further Remark \ref{rem-braid}(ii), where we recalled that the braid group of type $G(2,2,n)$ is the braid group of type $D_n$. In this case, it is easy to see that, in (\ref{quot-Y}), the relations expressing $s_0^2$ and $s_1^2$ are: $s_0^2=s_1^2=1+\frac{(q-q^{-1})}{2}(s_0+s_1)$. Using that $s_0$ and $s_1$ commutes, the fact that $s_0^2=s_1^2$ implies that $(s_0s_1^{-1})^2=1$. As a consequence, the theorem can be formulated in this case as follows: the algebra $Y_{2,n}^{\Z/2\Z}$ is the quotient of the braid group of type $D_n$ by the relations:
\[s_0^2=s_1^2=1+\frac{(q-q^{-1})}{2}(s_0+s_1)\ \ \ \text{and}\ \ \ s_i^2=1+\frac{(q-q^{-1})}{2}(S_i+s_i)\,,\ \ i=2,\dots,n-1\,,\]
where $S_1:=s_0$ and $S_{i}:=s_{i-1}s_{i}S_{i-1}s_{i}^{-1}s_{i-1}^{-1}$ for $i=2,\dots,n-1$.
 \end{itemize}
 \end{Rem}

\paragraph{Proof of the theorem.} As an intermediary step towards the proof of Theorem \ref{theo-def}, we provide another presentation of $Y_{d,n}^{\Z/p\Z}$. To do so, we will consider the algebra $\widetilde{Y}_{d,p,n}$ with generators  $\widetilde{g}_0$, $\widetilde{g}_1$, \ldots, $\widetilde{g}_{n-1}$ and $\widetilde{a}_0$ subject to the following relations :
\begin{itemize}
\item \rm{\textbf{(R'1)}}\ :\ \ \ $\left\{\begin{array}{rclcl}
\tg_i\tg_j & = & \tg_j\tg_i &\quad & \text{for $i,j\in\{1,\ldots,n-1\}$ such that $\vert i-j\vert > 1$\,,}\\[0.2em]
\tg_i\tg_{i+1}\tg_i & = & \tg_{i+1}\tg_i\tg_{i+1} && \text{for $i\in\{1,\ldots,n-2\}$\,,}\\[0.2em]
\tg_i^2  & = & 1 + (q-q^{-1}) \, \te_{i} \, \tg_i &\quad& \text{for $i\in\{1,\ldots,n-1\}$\,,}
\end{array}\right.$

\item \rm{\textbf{(R'2)}}\ :\ \ \   $\left\{\begin{array}{rclcl}
\tg_i\tg_0 & = & \tg_0\tg_i &\quad & \text{for $i\in\{3,\ldots,n-1\}$\,,}\\[0.2em]
\tg_0\tg_{2}\tg_0 & = & \tg_{2}\tg_0\tg_{2}\,, &&\\[0.2em]
\tg_0^2  & = & 1 + (q-q^{-1}) \, \te_{0} \, \tg_0\,, &&
\end{array}\right.$

\item \rm{\textbf{(R'3)}}\ :\ \ \  $\widetilde{g}_2 \widetilde{g}_0 \widetilde{g}_1^{-1}\widetilde{g}_2^{-1} \widetilde{g}_0 \widetilde{g}^{-1}_1=\widetilde{g}_0 \widetilde{g}^{-1}_1 \widetilde{g}_2\widetilde{g}_0 \widetilde{g}^{-1}_1 \widetilde{g}^{-1}_2$\ ,

\item \rm{\textbf{(R'4)}}\ :\ \ \  $\underbrace{\widetilde{g}_1 \widetilde{a}_0 \widetilde{g}^{-1}_0 \widetilde{g}_1 \widetilde{g}^{-1}_0 \widetilde{g}_1 \ldots}_{p+1 \text{ terms}} = \underbrace{\widetilde{a}_0 \widetilde{g}_0 \widetilde{g}^{-1}_1 \widetilde{g}_0 \widetilde{g}^{-1}_1 \ldots}_{p+1 \text{ terms}}$\ ,

\item \rm{\textbf{(R'5)}}\ :\ \ \   $\left\{\begin{array}{rclcl}
\widetilde{a}_0^{d/p} & = & 1\,, &\quad & \\[0.2em]
\widetilde{a}_0 \widetilde{g}_0 \widetilde{g}_1^{-1} & = & \widetilde{g}_0 \widetilde{g}_1^{-1} \widetilde{a}_0\,, &&\\[0.2em]
\widetilde{a}_0 \widetilde{g}_i  & = & \widetilde{g}_i \widetilde{a}_0 && \text{for $i\in\{2,\ldots,n-1\}$\,,}
\end{array}\right. $
\end{itemize}
where $\te_i$ for $i=0,1,\ldots,n-1$ are defined as follows :
\begin{equation}\label{tei}
\widetilde{e}_0=\widetilde{e}_1=\frac {1}{d}\sum_{1\leq s\leq d} (\widetilde{g}_0 \widetilde{g}^{-1}_1)^s\ \ \ \ \ \ \text{and}\ \ \ \ \ \ \widetilde{e}_i:=\widetilde{g}_{i-1}\widetilde{g}_i \widetilde{e}_{i-1} \widetilde{g}^{-1}_{i}\widetilde{g}^{-1}_{i-1}\ \ (i=2,\dots,n-1)\,.
\end{equation}

\begin{Prop}\label{prop-def}
The algebra $\widetilde{Y}_{d,p,n}$  is isomorphic to $Y_{d,n}^{\mathbb{Z}/p\mathbb{Z}}$\ .
\end{Prop}
\begin{proof}
We consider the following map from the set of generators of $Y_{d,n}^{\mathbb{Z}/p\mathbb{Z}}$ (from the presentation of  Lemma \ref{lem-pres}) to $\widetilde{Y}_{d,p,n}$:
$$\Phi \ :\ \ \ a_0\mapsto  \ta_0\,,\ \ \ \ \ a_1\mapsto\tg_0\tg_1^{-1}\,,\ \ \ \ \ \ g_i\mapsto\tg_i\,,\ i=1,\dots,n-1\,$$
and the following map from the set of generators of $\widetilde{Y}_{d,p,n}$ to $Y_{d,n}^{\mathbb{Z}/p\mathbb{Z}}$:
\[\Psi \ :\ \ \ \ta_0\mapsto a_0\,,\ \ \ \ \ \tg_0\mapsto a_1g_1\,,\ \ \ \ \ \ \tg_i\mapsto g_i\,,\ i=1,\dots,n-1\,.\]
We need to check that both maps extend to algebra homomorphisms. Then they will clearly be inverse to each other and thus provide the desired isomorphism. To do so, we use the presentation of $Y_{d,n}^{\mathbb{Z}/p\mathbb{Z}}$ given in Lemma \ref{lem-pres}.  We start with the following list of remarks:
\begin{itemize}
\item Relations (R'1) correspond to Relations (R1).
\item To see that Relation (R'4) corresponds to $g_1a_0g_1^{-1}=a_0a_1^p$\,, we add a factor $\tg_1^{-1}\tg_1$ in the left hand side of (R'4) to obtain $\tg_1\ta_0\tg_1^{-1}(\underbrace{\widetilde{g}_1 \widetilde{g}^{-1}_0 \widetilde{g}_1 \widetilde{g}^{-1}_0 \widetilde{g}_1 \ldots}_{p \text{ terms}})$.  This gives the desired relations thanks to the definitions of $\Phi$ and $\Psi$. 
\item  Relations (R'5) correspond to $a_0^{d/p}=1$, $a_0a_1=a_1a_0$ and $g_ia_0g_i^{-1}=a_0$ if $i\geq2$.
\item The first relation in (R'2) corresponds to $g_ia_1=a_1g_i$ if $i\geq3$. 
\item The second relation in (R'2) corresponds to $g_2a_1g_1g_2=a_1g_1g_2a_1g_1$. Now take the rightmost $g_1$ and send it to the left hand side. Then, using the braid relation between $g_1$ and $g_2$, one obtains $g_2a_1g_2^{-1}g_1g_2=a_1g_1g_2a_1$ which is equivalent to $g_2a_1g_2^{-1}=a_1a_2$.

\item Relation (R'3) corresponds to the fact that $g_2a_1g_2^{-1}$ commutes with $a_1$. Taking into account that $g_2a_1g_2^{-1}=a_1a_2$ (preceding item), it is equivalent to the fact that $a_1$ commutes with $a_2$.
\end{itemize}

As a consequence of the preceding list of remarks, to prove that $\Psi$ extends to a morphism, it remains only to check that the relation $\tg_0^2=1+(q-q^{-1})\te_0\tg_0$ is preserved. We thus need to show that  $(a_1g_1)^2=1+(q-q^{-1})e_1a_1g_1$. Using that $g_1a_1=a_1^{-1}g_1$, we see that the left hand side is $g_1^2$, and using that $a_1^d=1$, we see that $e_1a_1=e_1$. This shows the desired assertion.

Now, to prove that $\Phi$ extends to a morphism and to conclude the proof, it remains only to check that the relation $g_1a_1g_1^{-1}=a_1^{-1}$ is preserved by $\Phi$. We thus have to show that  $\tg_1\tg_0\tg_1^{-1}\tg_1^{-1}=\tg_1\tg_0^{-1}$, which is equivalent to
\begin{equation}\label{g0g1}
\tg_0^2=\tg_1^2\ .
\end{equation}
By (R'2), this is equivalent to the relation $\te_1\tg_0=\te_1\tg_1$. 
We have already seen above that (R'4) can be written $\tg_1\ta_0\tg_1^{-1}=\ta_0(\tg_0\tg_1^{-1})^p$. Together with $\ta_0^{d/p}=1$ and the fact that $\ta_0$ and $\tg_0\tg_1^{-1}$ commute, this yields $(\tg_0\tg_1^{-1})^d=1$. Using this, we have then $\te_1\tg_0\tg_1^{-1}=\te_1$ which gives $\te_1\tg_0=\te_1\tg_1$ and thus (\ref{g0g1}) is satisfied.
\end{proof}

\begin{proof}[End of the proof of Theorem \ref{theo-def}]
By Proposition \ref{prop-def}, we need to prove that the quotient of the group algebra of the braid group of type $G(d,p,n)$ by Relations (\ref{quot-Y}) is isomorphic to the algebra $\widetilde{Y}_{d,p,n}$. We will prove that an isomorphism is given by $s_i\leftrightarrow\tg_i$ ($i=0,1,\dots,n-1$) and $\alpha_0\leftrightarrow\ta_0$.

\vskip .2cm
First we will check that Relations (R'1)--(R'5) are satisfied for the generators $s_0,s_1,s_2,\alpha_0$ in the considered quotient of the braid group (for the rest of the proof, $s_0,s_1,s_2,\alpha_0$ will denote the generators in the quotient). A verification is needed only for Relations (R'3), (R'4) and the second line in (R'5).

First note that $(s_0s_1^{-1})^d=1$ implies that $e_1s_0s_1^{-1}=e_1$ and thus $e_1s_0=e_1s_1$. This in turn means $s_0^2=s_1^2$, or in other forms, 
\begin{equation}\label{s0s1}
s_0s_1^{-1}=s_0^{-1}s_1\ \ \ \ \ \text{and}\ \ \ \ s_1s_0^{-1}=s_1^{-1}s_0\ .
\end{equation}
Now we proceed in several steps:
\begin{itemize}
\item  We deal first with (R'4). We start with (Br4) and write it on the form
\[s_1\alpha_0s_1^{-1}\cdot(\underbrace{s_1s_0s_1\ldots}_{p \text{ terms}})\cdot(\underbrace{\ldots s_0^{-1}s_1^{-1}s_0^{-1}}_{p \text{ terms}})=\alpha_0\ .\]
Using (\ref{s0s1}), we can reorganise the position of the negative exponents to obtain $s_1\alpha_0s_1^{-1}(s_1s_0^{-1})^p=\alpha_0$. This gives Relation (R'4) for $\alpha_0,s_0,s_1$.
\item  Then we prove that $\alpha_0s_0s_1^{-1}=s_0s_1^{-1}\alpha_0$. The braid relation $\alpha_0s_0s_1=s_0s_1\alpha_0$ asserts that $\alpha_0$ commutes with $s_0s_1$. Therefore, in order to prove that $\alpha_0$ commutes with $s_0s_1^{-1}$, it is enough to prove that $\alpha_0$ commutes with $s_1^{-2}$ (or, equivalently, with $s_1^2$). To do so, we note that we have proved in the preceding item that $s_1\alpha_0s_1^{-1}=\alpha_0 (s_0s_1^{-1})^p$. Besides, using (\ref{s0s1}), we have $s_1\cdot s_0s_1^{-1}\cdot s_1^{-1}=s_1\cdot s_0^{-1}s_1\cdot s_1^{-1}=s_1s_0^{-1}=(s_0s_1^{-1})^{-1}$. Therefore, we have
\[s_1^2\cdot\alpha_0\cdot s_1^{-2}=s_1\cdot\alpha_0(s_0s_1^{-1})^p\cdot s_1^{-1}=\alpha_0(s_0s_1^{-1})^p(s_0s_1^{-1})^{-p}=\alpha_0\ .\]
This concludes the verification of $\alpha_0s_0s_1^{-1}=s_0s_1^{-1}\alpha_0$.
\item Finally, we prove (R'3) for $s_0,s_1,s_2$. We start from (Br3), which asserts that $s_0s_1$ commutes with $s_2s_0s_1s_2$, and we make the following reasoning, where $[x,y]$ denotes $xy-yx$:
\[\begin{array}{llll} [s_0s_1\,,\,s_2s_0s_1s_2]=0\ \ \ \ & \Rightarrow \ \ \ \ & [s_0s_1^{-1}\,,\,s_2s_0s_1s_2]=0\ \ & \text{( if $[s_1^{-2}\,,\,s_2s_0s_1s_2]=0$ )}\\[0.5em]
& \Rightarrow \ \ \ \ & [s_0s_1^{-1}\,,\,s_2s_0s_1^{-1}s_2^{-1}]=0\ \ & \text{( if $[s_0s_1^{-1}\,,\,(s_2s_1^2s_2)^{-1}]=0$ )\ .}
\end{array}\]
The last assertion is Relation (R'3) for $s_0,s_1,s_2$.

So we need first to check that indeed $s_1^{-2}$ (or, equivalently, $s_1^2$) commutes with $s_2s_0s_1s_2$. This follows from the following calculation
\[s_2s_0s_1s_2\cdot s_1=s_2s_0\cdot s_2\cdot s_1s_2=s_0\cdot s_2s_0s_1s_2\ ,\]
together with $s_0^2=s_1^2$, which was stated in (\ref{s0s1}).

Then it remains to check that $s_0s_1^{-1}$ commutes with $(s_2s_1^2s_2)^{-1}$ (or, equivalently, with $s_2s_1^2s_2$). In fact, both $s_0$ and $s_1$ commute with $s_2s_1^2s_2$. For $s_1$, it follows from the braid relation between $s_1$ and $s_2$. For $s_0$, it follows from the fact that $s_1^2=s_0^2$ (this is (\ref{s0s1}) again) and the braid relation between $s_0$ and $s_2$.
\end{itemize}

\vskip .2cm
Reciprocally, we need to prove that all the braid relations (Br1)--(Br5) together with Relations (\ref{quot-Y}) are satisfied by the generators of $\widetilde{Y}_{d,p,n}$. The only non-trivial relations to check are:
\[ (\tg_0\tg_1)^d=1\,,\ \ \ \ \ta_0\tg_0\tg_1=\tg_0\tg_1\ta_0\,,\ \ \ \ \underbrace{\widetilde{g}_1 \widetilde{a}_0 \widetilde{g}_0 \widetilde{g}_1 \widetilde{g}_0 \widetilde{g}_1 \ldots}_{p+1 \text{ terms}} = \underbrace{\widetilde{a}_0 \widetilde{g}_0 \widetilde{g}_1 \widetilde{g}_0 \widetilde{g}_1 \ldots}_{p+1 \text{ terms}}\ \ \ \text{and}\ \ \ (\tg_2\tg_0\tg_1)^2=(\tg_0\tg_1\tg_2)^2\ .\]
By Proposition \ref{prop-def} we can check these relations in $Y_{d,n}$, where $\ta_0,\tg_0,\tg_1,\tg_2$ correspond respectively to $t_1^p,t_1^{-1}t_2g_1,g_1,g_2$. This is immediate for the two first relations and straightforward for the two remaining ones. Alternatively, for the two last relations, one can use that $(\tg_0\tg_1)^d=1$ to obtain the analogues of (\ref{s0s1}) and then reverse the arguments used to prove (R'4) and (R'3) in the first part of the proof. We skip the details.
\end{proof}

\subsection{Parametrization of the simple $\mathbb{C}_{\theta} Y_{d,n}^{\mathbb{Z}/p\mathbb{Z}}$-modules}\label{action1}

Let  $\theta : \mathbb{C}[q,q^{-1}] \to \mathbb{C}$ be a specialization. We now   apply the results of Section \ref{clifford} in order to study the representation theory of  $\mathbb{C}_{\theta} Y_{d,n}^{\mathbb{Z}/p\mathbb{Z}}$.
 
 Let us fix $\ulambda \in \Lambda^e_{\mu}$ with $\mu\in \operatorname{Comp}_d (n)$. 
 For convenience we write $\ulambda=(\ulambda^1,\ldots,\ulambda^p)$ where each 
  $\ulambda^i$ consists in $d/p$-partitions of $\mu_{(d/p-1)i+1}+\ldots+\mu_{di/p}$.  
  We set $s:=s (\ulambda)$ to be the minimal integer such that for all $i\in \{1,\ldots, p\}$, we have $\ulambda^i=\ulambda^{i+s}$ (where the 
   indices are understood modulo $p$). 
   
   The inertia subgroup $\mathfrak{H}_{V^\ulambda} (\mathbb{Z}/p\mathbb{Z})$ is then 
    generated by $\sigma_{d/p}^s$,  where the permutation $\sigma_{d/p}$ of $\mathfrak{S}_d$ was introduced at the beginning of Section \ref{subsec-51}. So $\mathfrak{H}_{V^\ulambda} (\mathbb{Z}/p\mathbb{Z})$ is isomorphic to $\mathbb{Z}/(p/s) \mathbb{Z}$ and thus its irreducible representations are of all one dimensional and 
     naturally indexed by $\{1,\ldots,p/s\}$ 
     $$\textrm{Irr} (\mathfrak{H}_{V^\ulambda} (\mathbb{Z}/p\mathbb{Z}))
     =\{S^{i}\ |\ i=1,\ldots,p/s(\ulambda)\}$$
We are now ready to apply Proposition \ref{Cliff}.

\begin{Th}
The simple $\C_{\theta}Y_{d,n}^{\mathbb{Z}/p\mathbb{Z}}$-modules are given by the set 
$$\{ V^{\ulambda}_{S^i}\ |\ \ulambda \in \Lambda_{\mu}^e, \mu\in \operatorname{Comp}_d (n), i=1,\ldots,p/s(\ulambda)\}$$
and we have $V^{\ulambda}_{S^{i}}\simeq V^{\ulambda'}_{S^{i'}}$ if and only if 
     $\ulambda$ and $\ulambda'$ are in the same $\mathbb{Z}/p\mathbb{Z}$-orbit and $i=i'$. In addition, we have
     $$\operatorname{dim}_{\mathbb{C}} (V^{\ulambda}_{S^{i}}) =\frac{y (\ulambda,e)s(\ulambda)}{p}$$
\end{Th}   
    
\begin{Rem}

In the same spirit as in Section \ref{subsec-rep}, one can also show that the algebra $Y_{d,n}^{\mathbb{Z}/p\mathbb{Z}}$ 
 is not isomorphic to the usual cyclotomic Hecke algebra of type $G(d,p,n)$ (see the definition for example in \cite[\S 6.1]{CJ}), using the parametrization result  obtained in  \cite[Th. 6.4]{CJ} and looking at the specialization $\theta : \mathbb{C}[q,q^{-1}]\to \mathbb{C}$ such that $\theta (q)^2$ is a primitive root of order $d$.

\end{Rem}

\vspace{0.5cm}

\noindent {\bf Addresses:} \\

\noindent \textsc{Nicolas Jacon}, Universit\'e de Reims Champagne-Ardenne, UFR Sciences exactes et naturelles, Laboratoire de Math\'ematiques EA 4535
Moulin de la Housse BP 1039, 51100 Reims, FRANCE\\  \emph{nicolas.jacon@univ-reims.fr}\\
\\
\textsc{Lo\"\i c Poulain d'Andecy}, Universit\'e de Reims Champagne-Ardenne, UFR Sciences exactes et naturelles, Laboratoire de Math\'ematiques EA 4535
Moulin de la Housse BP 1039, 51100 Reims, FRANCE\\  \emph{loic.poulain-dandecy@univ-reims.fr}

\end{document}